\documentclass{article}

\usepackage{amsmath}
\usepackage{amscd}
\usepackage{amssymb}
\usepackage{theorem}
\usepackage{xypic}
\usepackage[all,v2]{xy}
\xyoption{2cell}
\xyoption{color}
\UseAllTwocells
\usepackage{enumitem}
\usepackage{rotating}
\usepackage{mathrsfs}
\usepackage{stmaryrd}
\usepackage{mathdots}
\usepackage{hyperref}
\usepackage{caption}
\usepackage[dvipsnames]{xcolor}
\usepackage{breqn}
\usepackage[bbgreekl]{mathbbol}
\usepackage{bbm}

\usepackage[bbgreekl]{mathbbol}
\usepackage{tikz-cd}

\newenvironment{items}
{\begin{enumerate}[topsep=3pt, itemsep=3pt, parsep=0pt, label=(\roman*)]}
{\end{enumerate}}

\newcommand{\smat}[1]{{\bigl(\begin{smallmatrix}#1\end{smallmatrix}\bigr)}}

\renewcommand{\tilde}{\widetilde}

\newcommand{\sep}{{\mbox{\tiny sep}}}

\newcommand{\taut}{{\mbox{\tiny taut}}}

\newcommand{\zz}{{\mathbb Z}}

\newcommand{\aaa}{{\mathbb A}}

\newcommand{\nn}{{\mathbb N}}
\newcommand{\nno}{{\mathbb N}_0}

\newcommand{\qq}{{\mathbb Q}}
\newcommand{\pp}{{\mathbb P}}
\newcommand{\cc}{{\mathbb C}}

\renewcommand{\O}{{\mathscr O}}

\newcommand{\tot}{\mathop{\rm tot}\nolimits}

\newcommand{\Bl}{\mathop{\rm Bl}\nolimits}

\newcommand{\epi}{\twoheadrightarrow}

\newcommand{\Isom}{\mathop{\rm Isom}\nolimits}

\newcommand{\GL}{\mathop{\rm GL}\nolimits}

\newcommand{\injectlim}{\mathop{\lim\limits_{\textstyle\longrightarrow}}\limits}
\newcommand{\projectlim}{\mathop{{\lim\limits_{\textstyle\longleftarrow}}}\limits}

\newcommand{\ol}{\overline}
\newcommand{\ul}{\underline}

\newcommand{\longiso}{\stackrel{\textstyle\sim}{\longrightarrow}}

\newcommand{\doublearrowstack}[2]%
                      {{{{\scriptstyle#1}\atop{\textstyle\longrightarrow}}\atop{{\textstyle\longrightarrow}\atop{\scriptstyle#2}}}}
\newcommand{\rightleftarrowstack}[2]%
                      {{{{\scriptstyle#1}\atop{\textstyle\longrightarrow}}\atop{{\textstyle\longleftarrow}\atop{\scriptstyle#2}}}}
\newcommand{\leftrightarrowstack}[2]%
                      {{{{\scriptstyle#1}\atop{\textstyle\longleftarrow}}\atop{{\textstyle\longrightarrow}\atop{\scriptstyle#2}}}}

\newcommand{\sfrac}[2]{{\textstyle\frac{#1}{#2}}}

\newtheorem{thm}{Theorem}[section]
\newtheorem{cor}[thm]{Corollary}
\newtheorem{lem}[thm]{Lemma}
\newtheorem{prop}[thm]{Proposition}

\theorembodyfont{\upshape}
\newtheorem{defn}[thm]{Definition}

\newtheorem{rmk}[thm]{Remark}

\newtheorem{ex}[thm]{Example}

\newenvironment{pf}{\begin{trivlist}\item[]{\sc Proof.}}%
            {\nolinebreak $\Box$ \end{trivlist}}

            {\nolinebreak $\Box$ \end{trivlist}}

\newcommand{\vir}{{\rm vir}}

\newcommand{\Gm}{{{\mathbb G}_m}}

\DeclareMathOperator\id{id}

\newcommand{\ch}{\mathop{\rm ch}\nolimits}

\newcommand{\Hilb}{\mathop{\rm Hilb}\nolimits}
\newcommand{\Sym}{\mathop{\rm Sym}\nolimits}

\newcommand{\noprint}[1]{}

\newcommand{\comment}[1]{$\mbox{}^{\spadesuit}${\marginpar{\footnotesize$\spadesuit$ #1}}}

\def\comment{\noprint}

\author{Kai Behrend}
\title{The Pardon Algebra for zero-cycles}

\begin{document}
\sloppy

\maketitle

\begin{abstract} 
Recently, John Pardon proved the MNOP conjecture by introducing a new mathematical gadget, which we call the Pardon homology algebra of 1-cycles in 3-folds.  We work out an analogous construction for $0$-cycles in $d$-folds.  This gives a new point of view on enumerative problems involving point-counting, such as, for example, the degree zero MNOP conjecture. 
\end{abstract}

\tableofcontents

\section*{Introduction}

This is an exploration of  John Pardon's theory~\cite{pardon}, in a much simpler case: instead of considering 1-cycles, we   consider 0-cycles. Pardon used his theory to prove the MNOP Gromov-Witten-Donaldson-Thomas correspondence.  We arrive at a new point of view on point enumeration problems. 

\paragraph{Summary of main result.} 
Fix $d$. A family of $d$-folds parametrized by $S$ is a separated (not necessarily proper) smooth morphism of schemes\comment{does it really not matter if we work with schemes or DM stacks?} $X\to S$ of relative dimension $d$.  We denote by $\ol Z_n X=X^n_S/\!/\Sigma_n\to S$ the relative $n$-th symmetric power of $X\to S$, and let $\ol ZX=\coprod_{n}\ol Z_nX$.  The Pardon homology algebra is by definition
$$H_\ast =\injectlim_{X/S} H_\ast(\ol ZX,\qq)\,.$$
Following Pardon, we construct a product and a coproduct on $H_\ast$, and  prove that $H_\ast$ is a graded commutative and cocommutative Hopf algebra over $\qq$. We define a `tautological' Hopf subalgebra, which contains all classes one might conceivably be interested in, and we prove that 
$$H_\ast^\taut=\qq[q_{n\lambda}]\,,$$
where $n$ runs over all positive integers, and $\lambda$ over all integer partitions with at most $d$ parts, such that $\lambda_i\geq n-1$ for all parts $\lambda_i$ of $\lambda$.
Here, $q_{n\lambda}$ is a class in $H_\ast$,  indexed in the colimit by the family of $d$-folds ($e_i$ is the $i$-th standard basis vector of $\qq^n$)
$$X=\tot\Big(\O(e_1)\oplus\ldots\oplus \O(e_d)\Big)\longrightarrow \pp^{\lambda_1}\times\ldots\times \pp^{\lambda_d}=S\,.$$
The $n$-fold zero section of this family defines a section $\Delta$ of $\ol Z_n X\to S$, and $q_{n,\lambda}=\Delta_\ast[\pp^{\lambda_1}\times\ldots\times\pp^{\lambda_d}]$. 

The interesting part of $H_\ast^\taut$, in the sense that it contains all $[\ol Z_nX]$ for compact $X\to\ast$, is the part where the homological degree equals $2nd$.  This forms a subalgebra (not a Hopf subalgebra), and is equal to 
$$\qq[q_{n,\lambda+n-1}]\,,$$
where $n$ runs over all positive integers, and $\lambda$ over all partitions of $d$.    (By $\lambda+n-1$ we denote the partition obtained by adding $n-1$ to each of the $d$ parts of $\lambda$). 

This reduces the computation of the values of any multiplicative point counting theory $e$ on all proper algebraic $d$-folds to evaluating $e$ on the classes $q_{n,\lambda+n-1}$. 

(The direct analogy to Pardon's work for 1-cycles on 3-folds 
would be the subalgebra of homological degree $0$, which is simply $\qq[q_{n,0}]$, where $n$ runs over the positive integers. So this algebra is freely generated by the `equivariant multiple points'.)

\paragraph{Outline.}

If $X$ is an algebraic $d$-fold, the space of $0$-cycles of degree $n$ is the symmetric product which we denote by $\ol Z_n X$.  In order to avoid dealing with singular varieties, we prefer working with the quotient stack $Z_nX=X^n/\Sigma_n$, of which $\ol Z_n X$ is the coarse moduli space. In terms of (co-)homology with $\qq$-coefficients, this makes essentially no difference. The moduli problem solved by $Z_nX$ is that of morphisms $D\to S\times X$, where $D\to S$ is finite \'etale of degree $n$. Thus, we can treat the theory of $0$-cycles also as a theory of maps. When passing from a fixed target variety $X$ to a family $X\to S$, we consider diagrams $D\to X\to S$, where $X\to S$ is a family of $d$-folds and  $D\to S$ is finite \'etale. We denote the relative space of maps of degree $n$ by $Z_nX=X^n_S/\Sigma_n$. 

For the theory to be compatible with passing from $Z_nX$ to its coarse moduli space $\ol Z_nX$, we need to make sure $Z_nX$ has finite stabilizers.  This requires $X\to S$ to be separated. 
On the other hand, dropping the separatedness assumption on our target families $X\to S$ gives rise to a perfectly consistent theory, which we develop in the first part of this paper.  It is simpler than the separated theory, deformation to the normal bundle is allowed, and multiple points can be separated into sums of simple points. The tautological Pardon algebra gets smaller, and is now equal  to $\qq[q_{1\lambda}]$. 

In the second part, we briefly introduce the dual of the Pardon homology algebra, the Pardon cohomology algebra.  It contains the enumerative theories against which classes in $H_\ast$ can be evaluated to get enumerative invariants. 

The third part contains the main results announced above.  As an illustration, we finish with a proof of the degree zero MNOP conjecture.  (Our proof apparently works for proper Deligne-Mumford  stacks, and it does not use relative Donaldson-Thomas theory.) 

Many problems involving point counts in $d$-folds can be formulated (and solved) using these results.  We expect that the calculations this gives rise to are essentially equivalent to well-known approaches to these kinds of problems, as developed for example by Ellingsrud-G\"ottsche-Lehn~\cite{Ellingsrudetal2001}.

All geometric arguments are purely algebraic.  This distinguishes this work from Pardon's for 1-cycles, which is essentially analytic.  Therefore, what we describe here should work for an algebraic (co)-homology theory such as   Chow-Witt (co)homology theory, which  is   homotopy invariant, and should give values in the Witt-ring.  This will be explored elsewhere.

\paragraph{Conventions.}
 All schemes and stacks are over $\cc$, all homology and cohomology coefficients are $\qq$. Deligne-Mumford stacks have separated diagonal, if the inertia stack is finite over the stack, the Deligne-Mumford stack is said to have finite stabilizers. 

\subsection*{Acknowledgements}
I want to thank Jim Bryan and Felix Thimm for many discussions about this and related material.
 
\section{The non-separated Pardon Hopf algebra}

\subsection{Homology and Cohomology}

We work with Deligne-Mumford stacks over the complex numbers.  They have representable, finite type and unramified diagonal, but we do not assume separatedness, or finiteness of stabilizers.  Singular homology and cohomology with coefficients in $\qq$ are defined via double complexes associated to groupoid presentations as in \cite{trieste}.  In addition, we assume existence of Gysin wrong way maps
$$f^!:H_i(Y)\longrightarrow H_{i+2d}(X)\,,$$
and 
$$f_!:H^i(X)\longrightarrow H^{i-2d}(Y)\,,$$
for any (not necessarily representable) proper lci morphism $f:X\to Y$ of relative dimension $d$, as constructed in~\cite{khan}.  These Gysin maps are  functorial, and commute with homology pushforward and cohomology pullback. Moreover, they are adjoint:
$$\langle e,f^!\alpha\rangle=\langle f_!e,\alpha\rangle\,,$$
for a cohomology class $e$ and a homology class $\alpha$. Finally, if $f:X\to Y$ is finite \'etale, then $f_\ast f^!:H_\ast(Y)\to H_\ast(Y)$ and $f_!f^\ast:H^\ast(Y)\to H^\ast(Y)$ are multiplication by the degree of $f$. 

For example, every smooth and proper Deligne-Mumford stack $X$ of dimension $d$ has a fundamental class $[X]\in H_{2d}(X)$, defined by $[X]=a^![\ast]$, where $[\ast]\in H_0(\ast)$ is the canonical element, and $a:X\to\ast$ the structure morphism. Dually, there is an integration map $\int_{[X]}:H^{2d}(X)\to\qq=H^0(\ast)$, defined by $a_!$. 

For example, if the finite group $G$ acts on the algebraic space $X$, then 
$$[X/G]=a^![\ast]=\frac{1}{|G|}\pi_\ast \pi^!a^![\ast]=\frac{1}{|G|}\pi_\ast[X]\,,$$
where $\pi:X\to X/G$ is the quotient map, and $a:X/G\to\ast$ the structure map. 
In the case where $X=\ast$,   we can identify $\qq=H_0(\ast)=H_0(BG)$, and via this identification, we have $[BG]=1/|G|$.  

\paragraph{K\"unneth theorem}  We will need to use the K\"unneth theorem. There are natural inverse isomorphisms 
$$\begin{tikzcd}
H_\ast(X)\otimes H_\ast(Y)\ar[r,"G"]&H_\ast(X\times Y)\ar[l,shift left=1ex,"F"]
\end{tikzcd}$$

\subsection{Families and cycle spaces}\label{non-sep}

\begin{defn}\label{def:fam}
We fix an integer $d\geq0$, and define a {\bf family of $d$-folds} (or simply a {\bf family}), to be any smooth morphism\comment{Note that we do not require representable. Should review everything to make sure we don't have hidden representability assumptions.}
 $X\to S$ of Deligne-Mumford stacks of relative dimension $d$. A {\bf morphism }of families from $Y\to T$ to $X\to S$ is a pullback diagram 
\begin{equation}\label{eq:fam}
\begin{tikzcd}
Y\ar[r]\ar[d] & X\ar[d]\\
T\ar[r] & S\rlap{\,.}
\end{tikzcd}\end{equation}
\end{defn}

\begin{defn}\label{def:mod}
For a family $X\to S$, we define the associated {\bf relative moduli stack} or {\bf cycle space} to be
$$ZX=\coprod_{n\geq0}Z_nX=\coprod_{n\geq0} X^n_S/\Sigma_n\,.$$
We hope suppressing the base $S$ in the notation of the moduli space will not lead to confusion.  Here, $X^n_S$ is the $n$-fold fibred product of $X$ relative to $S$, and $\Sigma_n$ is the symmetric group acting by permutations.  The quotient is the stack quotient. Note that we have $Z_0X=S$ (even if $X$ is empty), and $Z_1X=X$. 
\end{defn}

Every morphism of families (\ref{eq:fam}) gives rise to a cartesian diagram of moduli stacks
\begin{equation}\label{diag:mod}
\begin{tikzcd}
ZY\ar[r]\ar[d] & ZX\ar[d]\\
T\ar[r] & S\rlap{\,.}
\end{tikzcd}
\end{equation}

\begin{rmk}[Modular interpretation in terms of maps]
Given a family $X\to S$ and a morphism $T\to S$, the groupoid $ZX(T)$ of sections of $ZX\to S$ over $T$ is equivalent to the groupoid of commutative diagrams
$$\begin{tikzcd}
D\ar[r]\ar[d]&  X\ar[d]\\
 T\ar[r] & S\end{tikzcd}$$
 where $D\to T$ is finite \'etale and representable. Thus $ZX$ classifies maps from unordered finite sets to $X$. This is the point of view which we will usually take.
\end{rmk}

The universal map of degree $n$ of $X\to S$ is given by
$$\begin{tikzcd}
X_S^n\times_{\Sigma_n}\ul{n}\ar[r]\ar[dr] & Z_nX\times_S X\ar[d]\ar[r] & X\ar[d]\\
& Z_n X\ar[r] & S\rlap{\,.}
\end{tikzcd}$$
Note also that, for the universal map we have
$$X_S^n\times_{\Sigma_n}\ul n=Z_{n-1}X\times _SX\,.$$

\begin{rmk}
We allow non-separated families $X\to S$, so $Z_nX$ will typically not have finite stabilizers, and not have a good moduli space.  Because of this, calling $ZX$ the cycle space is an abuse of language. 
\end{rmk}

\begin{rmk}
Technically, the quotient stack classifies $S_n$-torsors $P\to S$, together with an $S_n$-equivariant map $f:P\to X_S^n$.  The moduli problem we listed is equivalent.  To a pair $(P,f)$ we associate the \'etale cover $P\times_{S_n}\ul{n}$ with the map to $X$ given by $[p,i]\mapsto f(p)_i$.  Conversely, to the \'etale cover $Z\to S$ and map $g:Z\to X$ we associate the $S_n$-torsor $\Isom_S(\ul n,Z)$ with the equivariant map $\phi\mapsto (\phi(i))_i$. 
\end{rmk}

\begin{rmk}[Smoothness]
For every family $X\to S$, and every $n\geq0$, the morphism $Z_nX\to S$ is smooth, of relative dimension $nd$. 
\end{rmk}

\begin{rmk}[Disjoint union of target]\label{rmk:prod}
Note that if $X\to S$ and $Y\to T$ are families, then for the cycle spaces we have
$$Z(X_T\amalg Y_S)=ZX\times ZY\,.$$
Here, $X_T\amalg Y_S$ the family parametrized by $S\times T$ obtained by taking the disjoint union of the pullbacks of the two families to the common base $S\times T$. 
\end{rmk}

\begin{rmk}[Addition]\label{rmk:plus}
Disjoint union of the source defines for every family $X\to S$ a morphism
\begin{align*}
+:ZX\times_S ZX&\longrightarrow ZX\\
(D\to X,D'\to X)&\longmapsto D\amalg D'\to X
\end{align*}
This morphism is finite, representable and \'etale, of degree $2^n$ over $Z_nX$. It turns $ZX\to S$ into a commutative monoid object over $S$.   The unit is given by $Z_0X=S$. 
\end{rmk}

\begin{rmk}[Compatibility]\label{rmk:add}
The addition operation is compatible with disjoint unions.  The following diagram is cartesian. (The top row is split into two.)
$$\begin{tikzcd}
(ZX\times_S ZX)\times (ZY\times_T ZY)\ar[dd,"+\times +"'] \ar[r,"\sim"] & (ZX\times ZY)\times_{S\times T}(ZX\times ZY)\ar[r,"\sim"]&\mbox{}\\\phantom{mmmmmmmmmmmmmmm}\ar[r,"\sim"] 
& Z(X_T\amalg Y_S)\times_{S\times T}Z(X_T\amalg Y_S)\ar[d,"+"]\\
ZX\times ZY\ar[r,"\sim"]& Z(X_T\amalg Y_S)
\end{tikzcd}$$
For example, in the case $S=T=\ast$, this means that given a cycle   $D\to X$, such that, at the same time, $D=E_1+E_2$, and $X=Y_1\amalg Y_2$, there exist unique cycles $E_{11}\to Y_1$ and $E_{21}\to Y_1$, as well as $E_{12}\to Y_2$ and $E_{22}\to Y_2$, such that both $E_{11}+E_{21}=D_1$ and $E_{12}+E_{22}=D_2$, as well as $E_{11}\amalg E_{12}=E_1$ and $E_{21}\amalg E_{22}=E_2$.  Here $D=D_1\amalg D_2$ is the decomposition of $D$ induced by the decomposition $X=Y_1\amalg Y_2$ of $X$. 
$$\begin{tikzcd}[row sep=0ex,column sep=1ex]
E_{11} & \amalg & E_{12} & = & E_1\\
+ & &+&&+\\
E_{21} & \amalg & E_{22} & = & E_2\\
\shortparallel &&\shortparallel &&\shortparallel\\
D_1&\amalg& D_2 &=& D
\end{tikzcd}$$
\end{rmk}

\subsection{Pardon homology} 

Fix $d$ throughout.

\begin{defn}
The {\bf homology Pardon algebra }is defined as
$$H_\ast=\injectlim_{X/S} H_\ast(ZX)\,.$$
Here, the colimit is taken over the category of families of $d$-folds, as defined in Definition~\ref{def:fam}. The transition maps in the colimit are the homomorphisms in homology induced by (\ref{diag:mod}). 
\end{defn}

Thus, given any family $X\to S$, and any homology class $\alpha \in H_\ast(ZX)$, we get an induced element in $H_\ast$ which we will also denote by $\alpha$. The only relations we introduce in $H_\ast$ are that if $X\to S$ is the pullback of another family $Y\to T$, via the cartesian morphism $f:X\to Y$, then $\alpha\in H_\ast(ZX)$ and $f_\ast\alpha \in H_\ast(ZY)$ define the same element in Pardon homology $H_\ast$. 

\paragraph{Grading}
We note that $H_\ast$ is graded by the degree $n$.  The direct sum decomposition
$$H_\ast(ZX)=H_\ast(\coprod_{n\geq0} Z_nX)=\bigoplus_{n\geq0} H_\ast(Z_nX)$$
survives the colimit. Of course, $H_\ast$ has a second grading, the homological grading, so $H_\ast$ is, in fact, bigraded.  When referring to the latter, we will always use the term `homological grading', to distinguish it from the former.

\paragraph{Multiplication}
We define multiplication in $H_\ast$ on representatives as follows. Let $\alpha\in H_\ast(ZX)$, for a family $X\to S$, and $\beta\in H_\ast(ZY)$ for a family $Y\to T$.  Then $\alpha\cdot\beta$ is associated with the disjoint union family $X_T\amalg Y_S$ over $X\times T$.  In fact, 
$$\alpha\cdot\beta=\alpha\times\beta \in H_\ast(ZX\times ZY)=H_\ast (Z (X_T\amalg Y_S))\,,$$
is the external product of $\alpha$ and $\beta$, via the isomorphism observed in Remark~\ref{rmk:prod}. 
This operation is compatible with the relations defining $H_\ast$, and thus passes to $H_\ast$. 

It is straightforward to prove that this multiplication is bigraded,  associative, and graded commutative with respect to the homological grading. A unit is provided by the empty family $\varnothing\to\ast$ parametrized by the point. 

\paragraph{Comultiplication}
We define the  comultiplication $\Delta:H_\ast\to H_\ast\otimes H_\ast$ on $H_\ast(ZX)$, for a family $X\to S$, as the composition
$$H_\ast(ZX) \stackrel{+^!}{\longrightarrow}H_\ast(ZX\times_S ZX)
\stackrel{\Delta_\ast}{\longrightarrow}H_\ast(ZX\times ZX) \stackrel{F}{\longrightarrow}H_\ast(ZX)\otimes H_\ast(ZX)\,.$$
Here, $F$ is  the Alexander-Whitney map, giving rise to   K\"unneth components. 
This map is compatible with the colimit definition and hence passes to $H_\ast$. 
Note that the $\Delta_\ast$ part of this definition doesn't change the class in $H_\ast$. 

The comultiplication is bigraded,  coassociative and graded cocommutative with respect to the homological grading.  The counit is given by the composition
\begin{equation}\label{eq:counit}
H_\ast(ZX)\stackrel{0^!}{\longrightarrow } H_\ast(S)\stackrel{\deg}{\longrightarrow}\qq\,,
\end{equation}
on the family $X\to S$. 

\paragraph{Graded Hopf algebra}

\begin{prop}
Pardon homology $H_\ast$ is a bigraded bialgebra.  It is graded commutative and cocommutative with respect to the homological grading.   
\end{prop}
\begin{pf}
As an example of the various properties, let us check that the coproduct is multiplicative. 
 Our claim is that
$$\Delta:H_\ast\longrightarrow H_\ast\otimes H_\ast$$
 respects multiplication. 
Thus, let $\alpha\in H_\ast(ZX)$ and $\beta\in H_\ast(ZY)$, for families $X\to S$ and $Y\to T$.  We claim
$$\Delta(\alpha\cdot \beta)=\Delta(\alpha)\cdot\Delta(\beta)\,.$$
Applying the cartesian diagram from Remark~\ref{rmk:add}, we see that
$$+^!(\alpha\cdot\beta)=  ( +^!\alpha) \times ( +^!\beta)\,.$$
Mapping to $ZX\times ZX\times ZY\times ZY=ZX\times ZY\times ZX\times ZY$, get
$$\Delta_\ast(+^!(\alpha\cdot\beta))= \Delta_\ast ( +^!\alpha) \times \Delta_\ast( +^!\beta)$$
Applying the Alexander-Whitney map $F$, we get
$$F(\Delta_\ast +^!(\alpha\cdot\beta))= F(\Delta_\ast  +^!\alpha) \times F(\Delta_\ast +^!\beta)\,,$$
which proves the claim.\comment{this is a rather abbreviated proof and obscures that we needed to use, for example, the compatibilty of $G$ with Gysin pullback}
\end{pf}

\begin{prop}
The bialgebra $H_\ast$ is connected, with respect to the grading given by cycle degree $n$.
\end{prop}
\begin{pf}
This basically follows from the fact that $Z_0(X/S)=S$, for any  family $X\to S$, even the empty one.

Let $X\to S$ be an arbitrary  family, and $\alpha\in H_\ast(Z_0 X)=H_\ast(S)$ a class of homological degree~$0$.  We also consider the empty  family $\varnothing\to S$. We have an isomorphism
$$H_\ast(Z_0\varnothing)=H_\ast(S)\,,$$
and so we can also think of $\alpha$ as an element $\alpha'\in H_\ast(Z_0\varnothing)$. 
In $H_\ast$ we have $\alpha'\sim\alpha$, as can be seen by considering the pullback diagram of  families
$$\begin{tikzcd}
X\ar[r]\ar[d] & X\times (\aaa^1\setminus0)\ar[d] & \varnothing\ar[d]\ar[l]\\
S\ar[r,"\id\times 1"] & S\times \aaa^1 &S\ar[l,"\id\times 0"']
\end{tikzcd}$$
with induced diagram of $Z_0$-spaces:
$$\begin{tikzcd}
S\ar[r]\ar[d] & S\times \aaa^1\ar[d] & S\ar[d]\ar[l]\\
S\ar[r,"\id\times 1"] & S\times \aaa^1 &S\ar[l,"\id\times 0"']
\end{tikzcd}$$
and exploiting homotopy invariance $H_\ast(S\times\aaa^1)=H_\ast(S)$\,. 

Now the empty family over $S$ pulls back from the empty family over $\ast$. Using this, we see that $\alpha'\sim 1$ in the Pardon ring.  Note that, in particular, all classes of degree zero have homological degree~$0$ as well. In other words, classes of degree 0 which have homological degree not equal to 0 vanish. 
\end{pf}

\begin{cor}
The bialgebra $H_\ast$ is a graded Hopf algebra. 
\end{cor}
\begin{pf}
This is proved in \cite{GrinbergReiner}, Proposition~1.4.16.  The antipode is, in fact, unique.
\end{pf}

\begin{cor}
Let $V\subset H_\ast$ be the $\qq$-vector space of primitive elements, i.e., those elements $p$, which satisfy $\Delta p=1\otimes p+p\otimes 1$.  Then 
$$H_\ast=\Sym V\,,$$
is the symmetric $\qq$-algebra generated by $V$. Multiplication is multiplication of symmetric tensors, comultiplication is uniquely determined by being multiplicative, namely
$$\Delta\prod_{i\in I}p_i=\sum_{I=I_1\amalg I_2}\Big(\prod_{i\in I_1}p_i\otimes\prod_{i\in I_2}p_i\Big)\,,$$
and the antipode is determined by being multiplicative, and multiplication by $-1$ on $V$. 

Thus, the structure of the Hopf algebra $H_\ast$ is completely determined by the structure of the $\qq$-vector space $V$ of primitive elements. 
\end{cor}
\begin{pf}
This follows from the theorem of Cartier-Gabriel (see \cite{primer}, Remark~3.8.3).  
\end{pf}

We will not be interested in the whole algebra $H_\ast$, instead we will define a `tautological' subalgebra, below.  

\subsection{Vertical classes}
There are two fundamental ways to construct classes in $H_\ast$. Thus, we obtain horizontal and vertical  classes.

If $X\to S$ is a proper  family, then $\pi:Z_nX\to S$ is a proper morphism, smooth of relative dimension $nd$. Hence Gysin pullback defines a map 
$\pi^!:H_\ast(S)\to H_{\ast+2nd}(Z_nX)$, and so every class $\alpha\in H_\ast(S)$ gives rise to a class $\pi^! \alpha\in H_{\ast+2nd}(Z_nX)$, which defines a class in $H_{\ast+2nd}$.   In particular, for $S=\ast$, we get the fundamental class $[Z_nX]\in H_{2nd}(Z_nX)$.  We call classes of the type $\pi^!\alpha$ {\bf vertical classes}.  We may sometimes use the notation $(Z_nX)^!\alpha=\pi^!\alpha$ for such a class. 

If $X\to S$ is a proper family, then for every point $s:\ast\to S$, the fibre $X_s$ of $X$ over $s$ gives rise to a class $[Z_nX_s]$ in Pardon homology.   If $S$ is connected, then all these vertical classes are equivalent in the Pardon homology algebra.  (This uses the fact that $H_0(S)=\qq$, for connected $S$.)

In particular, $[Z_nX]$ and $[Z_nY]$ (for proper $X$, $Y$ over $\ast$) are equal in $H_{2nd}$, if it is possible to put both $X$ and $Y$ in a proper family over a connected base. 

Some formulas come out nicer if we consider the formal sum 
$$(ZX)^!\alpha=\sum_{n\geq0} (Z_nX)^!\alpha\,.$$
For example, for a proper $X\to \ast$, we denote by $[ZX]$ the formal sum $\sum_n [Z_nX]$.

\begin{prop}[Product of vertical classes]
Let $X\to S$ and $Y\to T$ be proper families, and $\alpha\in H_\ast(S)$, $\beta\in H_\ast(T)$   homology classes. Then the induced vertical cycles obey
$$(ZX)^!\alpha\cdot(ZY)^!\beta= Z(X_T\amalg Y_S)^!(\alpha\times\beta)\,.$$

In particular, $[ZX]\cdot [ZY]= [Z(X\amalg Y)]$. 
\end{prop}
\begin{pf}
This follows from the compatibility of proper Gysin pullback with external homology products. 
\end{pf}

\begin{prop}[coproduct of vertical classes]
Let $X\to S$ be a proper family, and $\alpha\in H_\ast(S)$ a homology class. The coproduct of the vertical class $(ZX)^!\alpha$ is given by
$$\Delta((ZX)^!\alpha)=\sum_{\ell} (ZX)^!\alpha_\ell\otimes (ZX)^!\beta_\ell\,.$$
Here, $\Delta_\ast\alpha=\sum_\ell\alpha_\ell\otimes\beta_\ell$ is the K\"unneth decomposition of the diagonal of $\alpha$.

In other words, every family $X\to S$ defines a (formal) coalgebra morphism
\begin{align*}
H_\ast(S)&\longrightarrow H_\ast\\
\alpha&\longmapsto (ZX)^!\alpha\,.
\end{align*}

For example, 
$$\Delta[ZX]=[ZX]\otimes [ZX]\,.$$
\end{prop}

\subsection{Horizontal classes}

If we have a relative cycle family   $D\to X$  of degree $n$ in a (non-proper!)  family $X\to S$, it induces a section $\Delta:S\to Z^nX $ of $\pi:Z_nX\to S$.   Via $\Delta_\ast:H_\ast(S)\to H_{\ast}(Z_nX)$,
we get, for every $\alpha\in H_\ast(S)$, an element $\Delta_\ast\alpha$ in $H_\ast$, which we will denote  by $\eta(D/X/S,\alpha)$.
Such classes are {\bf horizontal classes}.

Two horizontal cycles $\eta(D/X/S,\alpha)$ and $\eta(E/Y/T,\beta)$ are equal in Pardon homology if there exists a family $M\to U$, and a relative cycle $F\to M\to U$, and morphisms $s:S\to U$ and $t:T\to U$, such that $D\to X\to S$ is the pullback of $F\to M\to U$ via $s$, and $E\to Y\to T$ is the pullback of $F\to M\to U$ via $t$, and $s_\ast\alpha=t_\ast\beta$ in $H_\ast(U)$. 

\begin{rmk}(Local nature of horizontal classes)\label{local:nature}
If we have a relative cycle family $D\to X\to S$, and $U$ is an open subset of $X$ through which $D$ factors, then for every $\alpha\in H_\ast(S)$, we have $\eta(D/X/S,\alpha)=\eta(D/U/S,\alpha)$. 
\end{rmk}
\begin{pf}
Consider the relative cycle family $D\times\aaa^1\to (X\times\aaa^1\setminus Z\times \{0\})\to S\times \aaa^1$, where $Z$ is the closed complement of $U$ in $X$. The pullback of this cycle family via $(\id,0):S\to S\times \aaa^1$ is $D\to U\to S$, and the pullback via $(\id,1):S\to S\times \aaa^1$ is $D\to X\to S$. Moreover, $(\id,0)_\ast\alpha=(\id,1)_\ast\alpha$ in $H_\ast(S\times \aaa^1)$ by homotopy invariance of homology.
\end{pf}

\begin{prop}[Product of horizontal cycles]\label{prodh}
The product of horizontal cycles $\eta(D/X/S,\alpha)$ and $\eta(E/Y/T,\beta)$ 
is the horizontal cycle
$$\eta(D/X/S,\alpha)\cdot\eta(E/Y/T,\beta)=\eta((D_T\amalg E_S)/(X_T\amalg Y_S)/S\times T,\alpha\times \beta)\,.$$
\end{prop}
\begin{pf}
Let us denote by $\delta:S\to ZX$ and $\gamma:T\to ZY$ the sections induced by $D\to X$ and $E\to Y$. Recall the diagram, where the two sections correspond:
$$\begin{tikzcd}[column sep=0em]
Z(X_T\amalg Y_S)\ar[rr,"\sim"]\ar[dr]&&ZX\times ZY\ar[dl]\\
&S\times T\ar[ul,bend left,"\delta_T\amalg\gamma_S"]\ar[ur,bend right,"\delta\times\gamma"']\rlap{\,.}\end{tikzcd}$$
We have 
$$\delta_\ast\alpha\cdot\gamma_\ast\beta=\delta_\ast\alpha\times\gamma_\ast\beta=(\delta\times\gamma)_\ast(\alpha\times\beta)=(\delta_T\amalg\gamma_S)_\ast(\alpha\times\beta)\,.$$
\end{pf}

\begin{prop}[Coproduct of horizontal cycles]\label{coprodh}
Consider a multisection 
$\delta_I:S\times I\to X$, where $I$ is a finite set of order $n$ and $S$ is connected. This gives a cycle family $S\times I\to X\to S$. Let $\alpha\in H_\ast(S)$, so that we have a horizontal class  $\eta(\delta_I/X/S,\alpha)$. Let $\Delta_\ast(\alpha)=\sum_\ell\alpha_\ell\otimes \beta_\ell$ be the K\"unneth decomposition of the diagonal of $\alpha$. Then
$$\Delta(\eta(\delta_I/X/S,\alpha))=\sum_{I=I_1\amalg I_2}\sum_\ell\eta(\delta_{I_1}/X/S,\alpha_\ell)\otimes \eta(\delta_{I_2}/X/S,\beta_\ell)\,.$$
In other words, we have a commutative diagram
$$\begin{tikzcd}
H_\ast(S)\ar[r,"\Delta_\ast"]\ar[d,"\eta({\delta_I}/X/S)"'] &
H_\ast(S)\otimes H_\ast(S)  \ar[d,"\sum_{I=I_1\amalg I_2}\eta(\delta_{I_1}/X/S)\otimes\eta(\delta_{I_2}/X/S)"]\\
H_\ast\ar[r,"\Delta"] &H_\ast\otimes H_\ast\rlap{\,.}
\end{tikzcd}$$
\end{prop}
\begin{pf}
This follows from the fact that the preimage of the section $\Delta_I:S\to ZX$ under the addition map $+:ZX\times_S ZX$ is the disjoint union of the sections $\Delta_{I_1}\times \Delta_{I_2}:S\to ZX\times_SZX$. The only way to decompose  $S\times I\to S$ into the union of two finite \'etale maps to $S$, is to decompose $I$ into subsets. 
\end{pf}

\begin{lem}\label{etale:pullback}
If we have a pullback diagram of  families, with $u$ finite \'etale  of degree $\delta$,
\begin{equation}\begin{tikzcd}\label{double}
X'\arrow[d]\arrow[r] & X\arrow[d]& Z_nX'\arrow[d]\arrow[r] & Z_nX\arrow[d,"\pi"]\\
S'\arrow[r,"u"] & S & S'\arrow[r,"u"] & S\end{tikzcd}\end{equation}
with induced pullback diagram of cycle spaces, and a cycle family  $D\to X$, then (with obvious notation)
$$\eta(D'/X'/S',\alpha')=\eta(D/X/S,u_\ast\alpha')$$
in $H_\ast$, 
and hence also
$$\eta(D'/X'/S',u^! \alpha)=\eta(D/X/S,u_\ast u^!\alpha)=
\delta\,\eta(D/X/S,\alpha)\,.$$
\end{lem}

\subsection{Pardon's trick.}
For most intents we are interested in vertical cycles, but horizontal cycles are easier to manipulate, especially because the assumption of properness on the family is dropped.  A key feature of the Pardon homology algebra is that it allows the conversion of vertical cycles into horizontal cycles.  This goes via passing to the universal situation.

In our simple situation, where cycle spaces are smooth, this trick is almost a tautology.

Let us suppose given a homology class $\alpha\in H_\ast(S)$, a proper family $X\to S$ with induced cycle space $\pi:Z_n X\to S$, and the induced vertical class $\pi^!\alpha\in H_\ast(Z_nX)$.   
We consider the universal cycle family in the family $X\to S$:
$$\begin{tikzcd}
D\ar[r]\ar[dr] & Z_nX\times_S X\ar[d] & Z_nX\times_S Z_nX\ar[d]\\
& Z_nX & Z_nX\arrow[u,bend left,"\Delta"]\rlap{\,.}\end{tikzcd}$$
Its classifying map $\Delta$ is the diagonal. 
We get the horizontal class
$$\Delta_\ast \pi^!\alpha\in  H_{\ast}(Z_nX\times_S Z_nX\to Z_nX)\,.$$
The claim is that 
$$\pi^!\alpha\sim \Delta_\ast \pi^!\alpha\,,$$
in the Pardon ring $H_\ast$. This can be seen by considering the pullback diagram of  families, and associated diagram of cycle spaces
$$\begin{tikzcd}
Z_nX\times_SX\arrow[r]\arrow[d]& X\arrow[d]& Z_nX\times_S Z_nX\arrow[r,"p"]\arrow[d]& Z_nX\arrow[d]\\
Z_nX\arrow[r,"\pi"] & S& Z_nX\arrow[r,"\pi"]\arrow[ur,"\id"]\arrow[u,bend left,"\Delta"] & S\end{tikzcd}$$
In fact, we have 
\begin{align*}
\Delta_\ast \pi^!\alpha\sim& p_\ast\Delta_\ast\pi^!\alpha\\
=&\id_\ast \pi^!\alpha\\
=&\pi^!\alpha\,.
\end{align*}
For example, (for proper $X\to \ast$) the (vertical) fundamental class $[Z_nX]$ is equal to the (horizontal) class of the diagonal in $H_\ast(Z_nX\times Z_nX)$.

\subsection{Tautological subalgebra}\label{taut}

We introduce some natural elements of the Pardon Hopf algebra.  They will be horizontal classes.  They will be indexed by $d\times n$-matrices of integers.  For later purposes, only the $n=1$ case will be important, as the corresponding classes suffice to multiplicatively generate the tautological subalgebra. 

Suppose given $n$ vector bundles of rank $d$ over $S$, denoted $E_1,\ldots,E_n$. We consider the disjoint union of total spaces of these bundles 
$$X=E_1\amalg\ldots\amalg E_n$$
as a family of $d$-folds over $S$. We also consider the zero sections $\delta_i:S\to E_i$. These give rise to a multisection $\delta$ of $X\to S$. So for $\alpha\in H_\ast(S)$, we have the horizontal class $\eta(\delta/X/S,\alpha)\in H_\ast$ of degree $n$.

The relative cycle space  
$$Z_nX=X^n_S/\Sigma_n=(E_1\amalg\ldots\amalg E_n)^n_S/\Sigma_n$$
has a distinguished component canonically isomorphic to the vector bundle
$$E_1\oplus\ldots\oplus E_n\longrightarrow S\,.$$
(It corresponds to those maps $D\to X$, such that for every $E_i$ the induced map to $E_i$ has degree 1.)
The classifying section $\Delta$ of the multisection $\delta$ is the zero section of this bundle. 
$$\begin{tikzcd}[column sep=1ex]
E_1\oplus\ldots\oplus E_n\ar[rr,"\iota"]\ar[dr] && Z_nX\ar[dl]\\
&S\ar[ul,bend left,"0"]\ar[ur,bend right,"\Delta"']
\end{tikzcd}$$

Now consider a $d\times n$ matrix $N=(n_{ij})$ of non-negative integers and specialize further to
$$S=\pp=\prod_{\substack{i=1..d\\j=1..n}}\pp^{n_{ij}}\,$$
and
$$E_j=\bigoplus_{i=1..d}\O_\pp(i,j)\,,$$
where $\O_\pp(i,j)$ is the bundle $\O_{\pp^{n_{ij}}}(1)$ pulled back to $\pp$ via the projection into the $(i,j)$-component.
Moreover, we let $\alpha$ be the fundamental class
$$\alpha=[\pp]\in H_{2|N|}(\pp)\,,$$
where $|N|=\sum n_{ij}$. 
We will denote the induced horizontal class by 
$$q_N=\eta(\delta/X/\pp,[\pp])\,.$$

Of particular importance will be the case $n=1$, where $N$ is a matrix with one column, which we will denote by a column vector $\vec n$.  So the tautological class $q_{\vec n}$ is associated to the family $$\O(1,0,\ldots,0)+\ldots+\O(0,\ldots,0,1)\longrightarrow \pp^{n_1}\times\ldots\times\pp^{n_d}\,,$$ and the degree 1 map given by the zero section (where $Z_1X=X$), and the homology class $[\pp^{n_1}]\times\ldots\times[\pp^{n_d}]$. 

\begin{rmk}
We have the obvious relation
$$q_N=q_M$$
if $N$ is obtained from $M$ by permuting the columns. 
\end{rmk}

\begin{prop}
For the product of the tautological classes $q_N$ and $q_M$, we have
$$q_N\cdot q_M=q_{N\amalg M}\,,$$
where we write $N\amalg M$ for a matrix obtained by taking all columns from both $N$ and $M$ and forming a new matrix. (This follows from Proposition~\ref{prodh}.)

For the coproduct of the tautological class $q_N$, we have
 $$\Delta(q_N)=\sum_{N=P\amalg Q}q_P\otimes q_Q\,,$$
where the sum is over all ways to split the set of columns of $N$ into two subsets.

It follows that the $\qq$-vector space spanned by the $q_N$, for all $n$, and all $d\times n$-matrices $N$ of non-negative integers, is a Hopf subalgebra of $H_\ast$. 
\end{prop}
\begin{pf}
Let us prove the formula for the coproduct.  We need some more notation.  Let  us write $\pp^N$ for the product of all $\pp^{n_{ij}}$, where $N=(n_{ij})$.  Let us write $J$ for the indexing set of the columns of $N$. For a subset $J_1\subset J$, let us write 
$$E_{J_1}=\coprod_{j\in J_1} E_j\,$$
and $\delta_{J_1}$ for the multi-$0$-section $\pp^N\times J_1\to E_{J_1}$.

The K\"unneth decomposition of the  diagonal of $\pp^N$ is given by
$$\Delta_\ast[\pp^N]=\sum_{N=N_1+N_2}[\pp^{N_1}]\otimes[\pp^{N_2}]\,.$$
So by Proposition~\ref{coprodh}, we have
\begin{align*}
\Delta(q_N)=&\Delta(\eta(\delta_J/E_J/\pp^N,[\pp^N]))\\
=&\sum_{J=J_1\amalg J_2}\sum_{N=N_1+N_2}\eta(\delta_{J_1}/E_J/\pp^N,[\pp^{N_1}])\otimes
\eta(\delta_{J_2}/E_J/\pp^N,[\pp^{N_2}])\\
=&\sum_{J=J_1\amalg J_2}\sum_{N=N_1+N_2}\eta(\delta_{J_1}/E_{J_1}/\pp^N,[\pp^{N_1}])\otimes
\eta(\delta_{J_2}/E_{J_2}/\pp^N,[\pp^{N_2}])\,,
\end{align*}
where we have used the local nature of horizontal classes in the last step. Now we claim that 
$\eta(\delta_{J_1}/E_{J_1}/\pp^N,[\pp^{N_1}])$ vanishes unless $N_1=N|_{J_1}\amalg \,0|_{J_1^c}$. In fact, this follows from the fact that $E_{J_1}\to \pp^N$ pulls back from $\pp^{N|_{J_1}}$, and $[\pp^{N_1}]$ pushes forward to $0$ in $H_\ast(\pp^{N|_{J_1}})$, unless $N_1|_{J_1}=N|_{J_1}$. Thus,
$$\Delta(q_N)=\sum_{J=J_1\amalg J_2} \eta(\delta_{J_1}/E_{J_1}/\pp^{N|_{J_1}},[\pp^{N|_{J_1}}])\otimes
\eta(\delta_{J_2}/E_{J_2}/\pp^{N|_{J_2}},[\pp^{N|_{J_2}}])\,,$$
which is our claim.
\end{pf}

\begin{defn}
We call the $\qq$-vector subspace of $H_\ast$ generated by the $q_N$ the   {\bf tautological }Pardon-Hopf algebra (of finite maps to $d$-folds).
\end{defn}

\begin{rmk}
This is, indeed, a Hopf algebra, it is bigraded, although all homological degrees are even, so this Hopf algebra is commutative and cocommutative without the `graded' qualifier. It will turn out that the space of primitives is equal to the subspace of cycle degree~1.
\end{rmk}

Note that  the $d\times 1$-matrices with non-negative integer entries form a multiplicative generating set for the tautological algebra.  Let us write $d\times 1$-matrices as column vectors $\vec n$, and so $q_{\vec n}$ for these multiplicative generators.  There is exactly one $d\times 0$ matrix, namely the empty one.  The associated tautological horizontal class is $1\in H_\ast$, the multiplicative unit.  Taking this into account, we see that the multiplicative generators $q_{\vec n}$ are primitive (this is anyway clear, because degree one elements are automatically primitive). 

\begin{prop}\label{runs:over}
The $\qq$-vector space $P$ of primitive elements in the tautological algebra is spanned by the $q_{\vec n}$, where $\vec n$ runs over the column vectors of size $d$, with non-negative integer entries. The symmetric group on $d$ letters $\Sigma_d$ acts on these generators by permuting the entries of 
$\vec n$. The space $P$ is isomorphic to the quotient of the free $\qq$-vector space on the generators 
$\vec n\in \zz_{\geq0}^n$ modulo this action of $\Sigma_d$.  We can canonically identify $P$ with the homology of $B\GL_d$ (or the graded dual of the cohomology of $B\GL_d$). A basis of $P$ is indexed by integer partitions with $d$ parts (with trailing zeros allowed), where we set $q_\lambda=q_{\vec n}$, where $\vec n$ is the vector containing the parts of $\lambda$ in decreasing order.
\end{prop}
\begin{pf}
There cannot be any primitives which are not in the span of the $q_{\vec n}$, because if there were, they would not be in the subalgebra generated by the $q_{\vec n}$.  But this subalgebra is the entire tautological algebra.   

Permuting the entries of $\vec n$ gives an isomorphic family of maps, and hence an equal element in $H_\ast$. It remains to show that the $q_\lambda$ are $\qq$-linearly independent. For this, we borrow the notion of enumerative theory from the next section.   

Let $m_\lambda$ be the monomial symmetric function given by the partition $\lambda$, and $e_{\lambda}$ the enumerative theory (of degree 1) it defines, see \ref{cherntang}.

Denote $E=\O(1,0,\ldots,0)+\ldots+\O(0,\ldots,0,1)$, and the bundle projection by $\pi:E\to \pp^{n_1}\times\ldots\times\pp^{n_d}$. Also, denote the first Chern class of the tautological bundle coming from the $i$-th factor of $\pp^{n_1}\times\ldots\times\pp^{n_d}$ by $\gamma_i$. So $c_p(E)$ is the elementary symmetric function $\sigma_p$ of degree $p$ in $\gamma_1,\ldots,\gamma_d$. 

When we pair a monomial $\gamma_1^{r_1}\ldots\gamma_d^{r_d}$ with $[\pp^{n_1}]\times\ldots\times[\pp^{n_d}]$, we obtain $1$ if all $r_i=n_i$, and $0$ otherwise. 

We pair $e_\lambda$ with $q_{\vec n}$. \begin{align*}
\langle e_\lambda,q_{\vec n}\rangle=&\langle m_\lambda(\pi^\ast E),0_\ast([\pp^{n_1}]\times\ldots\times [\pp^{n_d}])\rangle\\
=&\langle \pi^\ast m_\lambda(E),0_\ast([\pp^{n_1}]\times\ldots\times [\pp^{n_d}])\rangle\\
=&\langle  m_\lambda(E),[\pp^{n_1}]\times\ldots\times [\pp^{n_d}]\rangle\\
=&\text{  coefficient of $\gamma_1^{n_1}\ldots\gamma_d^{n_d}$ in $m_\lambda(\gamma)$}\,.
\end{align*}
This is equal to 1 if the partition defined by $\vec n$ is equal to $\lambda$, and zero otherwise. 
\end{pf}

\begin{ex}
For $d=0$, there is only the empty partition, so the tautological algebra is
$$\qq[q_{\varnothing}]\,.$$
For $d=1$, partitions are non-negative integers, so the tautological algebra is
$$\qq[q_0,q_1,q_2,\ldots]\,.$$
For $d=2$, the tautological algebra is
$$\qq[q_{\smat{0\\0}},q_{\smat{1\\0}},q_{\smat{2\\0}},q_{\smat{1\\1}},\ldots]$$
\end{ex}

\begin{ex}\label{dim0}
Let us consider the case $d=0$. In this case a family $X\to S$ is an \'etale morphism of Deligne-Mumford stacks. A degree $n$ map to $X\to S$ is a morphism $D\to X$, such that $D\to S$ is finite \'etale.  Families over the point are finite sets, and for a finite set $X$ of order $k$,  the cycle space $Z_nX$ is 
$$Z_nX= X^n/\Sigma_n=\coprod_{n=n_1+\ldots+n_k}B\Sigma_{n_1}\times\ldots \times B\Sigma_{n_k}\,,$$
and
$$ZX=\Big(\coprod_{n} B\Sigma_n\Big)^X\,.$$

The tautological classes $q_N$ are indexed by $0\times n$-matrices, which are necessarily all empty, so there is only one per degree $n$, the empty one.  We will denote the corresponding generator by $q_n$. It is associated to the finite map $\ul n\stackrel{\id}{\longrightarrow}\ul n\to \ast$, and is equal to $1$ in the copy of $\qq$ corresponding to the distingished component in the cycle space $Z_n \ul n=\ul n^n/\Sigma_n$. As we have established, we have $q_n=q_1^n$, and so the tautological algebra in dimension $0$ is generated multiplicatively by $q_1$. 

We have
$$\Delta(q_n)=\sum_{i=0}^n \binom{n}{i}  q_{i}\otimes q_{j}\,.$$

If we restrict to the family $X=\ast\to S=\ast$, then the cycle space is $ZX=\coprod_n B\Sigma_n$, and the (vertical) fundamental class is $[Z_nX]=1/n!\in H_0(B\Sigma_n)=\qq$.  The horizontal class associated to the finite map $\ul n\to \ast \to \ast$, denoted $\eta(\ul n/\ast/\ast,[\ast])$ is equal to $1\in H_0(B\Sigma_n)=\qq$.

Consider the affine line with $n$ origins $\tilde \aaa^1$.  It has an \'etale projection to $\aaa^1$, which is the identity away from the origins, and maps all $n$ origins of $\tilde \aaa^1$ to the single origin in $\aaa^1$. There is also a morphism $\aaa^1\times \ul n\to \tilde \aaa^1$, which maps the $i$-th origin in $\aaa^1\times \ul n$ to the $i$-th origin in $\aaa^1$, for all $i=1,\ldots,n$, and is the projection onto $\aaa^1$ away from the origins.  In fact, the composition $\aaa^1\times \ul n\to \tilde \aaa^1\to \aaa^1$ is a family of finite maps of degree $n$ parametrized by $\aaa^1$. We see that horizontal classes $\eta(\ul n/\ast/\ast,[\ast])$ and $\eta(\ul n/\ul n/\ast,[\ast])$ are equal in $H_0$. (In fact, $\tilde\aaa^1$ is the deformation to the normal cone of the map $\ul n\to \ast$.)
We have proved that $\eta(\ul n/\ast/\ast,[\ast])=q_n=q_1^n$.  From this, it follows that 
$[Z_n\ast]=\frac{1}{n!}\eta(\ul n/\ast/\ast,[\ast])=\frac{1}{n!}q_1^n$, and so $[Z\ast]=e^{q_1}$, and 
finally
$$[Z X]=(e^{q_1})^{\chi (X)}\,.$$
We see that all vertical classes are tautological. 

In fact, the tautological algebra is equal to the polynomial ring $\qq[q_1]$,  because it is graded and $q_n$ is non-zero, as can be seen from the enumerative theory (see below),  which is $1\in H_0(Z_nX)$ for all $X/S$ and $0\in H_0(Z_mX)$, for all $X/S$ and all $m\not=n$. 
\end{ex}

\begin{ex}
We define the {\bf total degree }of an element in $H_\ast$ represented by $\alpha\in H_i(Z_nX)$, for a family $X\to S$, by $\deg(\alpha)=i-2nd$.  The motivation for doing this is that all vertical fundamental classes $[Z_nX]$ for $X\to \ast$ have total degree zero.  We denote the subset of $H_\ast$ of elements of degree $0$ by $H_\ast^0$. This is a subalgebra, but not a Hopf subalgebra, because the comultiplication does not preserve $H_\ast^0$.  The total degree of the tautological class $q_N$ is 
$$\deg(q_N)= 2\Sigma(N)-2nd\,, $$
where $\Sigma(N)$ denotes the sum of all entries of $N$.  So $\deg(q_N)=0$ if and only if the entries of $N$ average out to $1$. Unfortunately, a product can be of degree zero without the factors having this property.  Therefore it is not straightforward to write down a presentation of this algebra as a polynomial ring.

In dimension 1, the tautological degree zero algebra is
$$\qq[q_1,q_0q_2, q_0^2q_3,\ldots]\,.$$
\end{ex}

\subsection{Deformation argument}

\begin{thm}
The tautological algebra contains all vertical and all horizontal classes.
\end{thm}
\begin{pf}
We already know that all vertical classes are horizontal, and so it suffices to consider a horizontal class $\eta(D/X/S,\alpha)$, where $D\to X\to S$ is a family of finite maps, say of degree $n$,  to the family $X\to S$, and $\alpha\in H_\ast(S)$. As before, we will denote the section of $Z_nX\to S$ defined by $D$ by $\Delta$, so that $\eta(D/X/S,\alpha)=\Delta_\ast\alpha\in H_\ast(Z_nX)$.

\paragraph{Step 1. Finite \'etale base change.}  First we consider the base change $f:D\to S$, which is a finite \'etale map of degree $n$. Lemma~\ref{etale:pullback} tells us that we can replace $\alpha$ by $\frac{1}{n}f^!\alpha$,  and replace the base $S$ by $D$. This  simplifies the problem, because now $D\to S$ has obtained a section, and so splits off a copy of $S$.  Repeating this argument, we reduce to the case where $D$ is the disjoint union of $n$ copies of $S$. 

\paragraph{Step 2. Deformation to the normal cone.}
This is another way to simplify horizontal classes.   
We consider the deformation to the normal cone
$$\begin{tikzcd}
D\ar[r]\ar[d] & D\times\pp^1\ar[d]\ar[r] & D\ar[l]\ar[d]\\
C_{D/X}\ar[r]\ar[d] & M^0_{D/X}\ar[d] & X\ar[l]\ar[d]\\
S\ar[r,"\id\times\infty"] & S\times \pp^1 & S\ar[l,"\id\times 0"']
\end{tikzcd}$$
of the morphism $D\to X$ relative to $S$.
The deformation space is $M^0_{D/X}$, it is flat over $S\times \pp^1$.  So to check that it is smooth over $S\times \pp^1$, it suffices to check fiberwise. But as $D\to X$ is a regular local immersion,  the normal cone $C_{D/X}$ is a vector bundle over $D$, and $D\to S$ is smooth, so, indeed, $C_{D/X}\to S$ is smooth.

Note that if $D\to X$ is not an immersion, the deformation space $M^0_{D/X}$ is typically not separated. (See Example~\ref{dim0}.) This is why we insist that families be allowed to be non-separated. 

Anyway, $M_{D/X}^0\to S\times \pp^1$ is a family, and $D\times \pp^1\to M^0_{D/X}$ is a family of finite maps of degree $n$.

We get the induced pullback diagram of relative cycle spaces
$$\begin{tikzcd}
Z_n C_{Z/X}\ar[d]\ar[r] & Z_n M_{D/X}^0\ar[r]\ar[d] & Z_nX\ar[l]\ar[d]\\
S\ar[r]\ar[u,bend left,"\Delta_\infty"] & S\times \pp^1\ar[u,bend left,"\Delta"] & S\ar[l]\ar[u,bend left,"\Delta_0"]
\end{tikzcd}$$
Both horizontal classes $\Delta_{0\ast}\alpha$ and $\Delta_{\infty\ast}\alpha$ push forward to $\Delta_\ast(\alpha\times[\ast])$, so they are equal in Pardon homology.

This allows us to assume that the family $X\to S$ is a vector bundle over $D$, composed with the finite \'etale projection $D\to S$. 

\paragraph{Step 3. Universal bundles.}
Let us consider a special case. Suppose given $n$ vector bundles of rank $d$ over $S$, denoted $E_1,\ldots,E_n$. We consider the disjoint union of total spaces of these bundles 
$$X=E_1\amalg\ldots\amalg E_n$$
as a family of $d$-folds over $S$. We also consider 
$$D=\underbrace{S\amalg\ldots\amalg S}_{\text{$n$ times}}$$
and the map $D\to X$ which maps the $i$-th copy of $S$ to the zero section in $E_i$. We have the associated classifying section $\Delta:S\to Z_nX$.  

The relative cycle space 
$$Z_nX=X^n_S/\Sigma_n=(E_1\amalg\ldots\amalg E_n)^n_S/\Sigma_n$$
has a distinguished component canonically isomorphic to the vector bundle
$$E_1\oplus\ldots\oplus E_n\longrightarrow S\,.$$
The classifying section $\Delta$ is the zero section of this bundle. 
$$\begin{tikzcd}[column sep=1ex]
E_1\oplus\ldots\oplus E_n\ar[rr,"\iota"]\ar[dr] && Z_nX\ar[dl]\\
&S\ar[ul,bend left,"0"]\ar[ur,bend right,"\Delta"']
\end{tikzcd}$$
We consider the horizontal class $\Delta_\ast\alpha=\iota_\ast  0_\ast\alpha$, for 
$\alpha\in H_\ast(S)$.

By the splitting principle, we can assume that each $E_j$ is a sum of $d$ line bundles,
$$E_j=L_{1j}\oplus\ldots\oplus L_{dj}\,.$$
(Make the base change to a $\GL_d/T$-bundle, this is surjective on homology. This follows from the Leray-Hirsch Theorem.)

Let us deal with a single line bundle $L$ over $S$. We let $k> \frac{1}{2}\deg\alpha$, and construct the diagram
$$\begin{tikzcd}
L^{\oplus k+1}\setminus S\ar[r,"g"]\ar[d,"f"'] & \pp^k\\
S\rlap{\,.}
\end{tikzcd}$$
The map $f$ is the projection, the map $g$ is 
$$g(x_0,\ldots,x_k)=\langle x_0,\ldots,x_k\rangle\,.$$
We have $f^\ast L=g^\ast \O_{\pp^k}(1)$. 
The codimension of $S$ in $L^{\oplus k+1}$ is $k+1$. From this it follows that  the homology class $\alpha\in H_\ast(S)$ lifts to $L^{\oplus j+1}\setminus S$, as long as $k\geq \frac{1}{2}\deg\alpha$.  Thus, making the base change from $S$ to $L^{k+1}\setminus S$, we can assume that $L=g^\ast\O(1)$, for a morphism $g:S\to \pp^k$, as long as $k> \frac{1}{2}\deg\alpha$. 

Circling back to the $d\times n$-matrix of line bundles $L_{ij}$, we can assume that every one of these line bundles is the pullback of $\O(1)$, via a map $g_{ij}:S\to\pp^k$.  We just need that $k> \frac{1}{2}\deg\alpha$. 

Putting all these maps together, we obtain a map
$$g:S\longrightarrow (\pp^k)^{dn}$$
such that the line bundle $L_{ij}$ is the pullback via $g$ of the bundle $\O(i,j)$ on $(\pp^k)^{dn}$. Here we are using notation as in Section~\ref{taut}.

Pushing forward via $g$, we may assume that $S=(\pp^k)^{dn}$ and the bundles are $\O(i,j)$.  

Now the homology of $(\pp^k)^{nd}$  is generated (as a $\qq$-vector space) by the fundamental classes of linearly embedded products of projective space $\prod_{ij}\pp^{n_{ij}}$.  We reduce to the case where 
$$S=\pp=\prod_{\substack{i=1..d\\j=1..n}}\pp^{n_{ij}}\,,$$
and
$$\alpha=[\pp]\,,$$
where
$$X=\bigoplus_{j=1..d}\O_\pp(j,1)\amalg\ldots\amalg\bigoplus_{j=1..d}\O_\pp(j,n)\,$$
and 
$$D=\underbrace{S\amalg\ldots\amalg S}_{\text{$n$ times}}\,,$$
and $D\to X$ is the multi-zero section.The horizontal class is $\Delta_\ast[\pp]$. 

Combining Steps 1, 2, and 3 we obtain the desired proof. 
\end{pf}

\section{Enumerative theories}

\begin{defn}
  The {\bf cohomology Pardon algebra }is defined as 
$$H^\ast= \projectlim_{X/S} H^\ast(ZX)\,,$$
the connecting maps being cohomology pullback.  Following Pardon, we call elements of $H^\ast$, {\bf enumerative theories}. 

Thus,  
an enumerative theory $e\in H^\ast$,  associates to every   family of $d$-folds $X\to S$, and every integer $n$, a cohomology class 
$$e(n,X/S)\in H^{\ast}(Z_nX)\,.$$
The only condition is that this class is natural with respect to pullbacks of   families.   

Enumerative theories are formal sums: $e(X/S)=\sum_{n\geq0}e(n,X/S)$. If $e(n,X/S)$ is the only non-zero term in this sum, $e(X/S)$ is {\bf homogeneous }of degree $n$. 
\end{defn}

\begin{ex}[Chern classes of tangent bundle]\label{cherntang}
Let  $n$ be a non-negative integer, and  $P$ be a symmetric function.   We define a homogeneous  enumerative theory of degree $n$ by 
$$e_P(n,X/S)=P(T_{Z_nX/S})\in H^\ast(Z_nX)\,.$$
(We substitute the Chern roots of the relative tangent bundle of $Z_nX\to S$ into the symmetric function $P$.)
This defines an enumerative theory, because relative tangent bundles pull back under pullback of families, and their cycle spaces, and Chern classes are natural, as well.  

Joining all the $e_P(n)$ together, we obtain $e_P=\sum_n e_P(n)$. 
\end{ex}

\begin{ex}[Powers of total Chern class and Euler class]\label{ex:pow}
Define 
$$c^k(X/S)=c(T_{ZX/S})^k\in H^\ast(ZX)\,.$$
Here, $c$ denotes the total Chern class. 
Write
$$e^k(X/S)=e(T_{ZX/S})\in H^\ast(ZX)\,,$$
where $e$ denote the Euler class.

Both $c^k$ and $e^k$ are enumerative theories. 
\end{ex}

\begin{ex}\label{is:not}
Let $d=1$, so that $Z_nX$ has a   coarse moduli space $\ol Z_nX$, which is smooth over $S$. 
Then 
$$e(T_{\ol Z_nX})\in H^\ast(\ol Z_nX)=H^\ast(Z_nX)$$
is {\em not }an enumerative theory, because the formation of the coarse moduli space does not commute with pullbacks.  For it to commute with pullbacks, we would need our cycle stacks to have finite stabilizers, which would require the families to be separated. 
\end{ex}

\paragraph{Pairing with homology.} There is a natural pairing 
$$H^\ast\otimes_\qq H_\ast\longrightarrow\qq\,,$$
induced by the  pairing  
$$H^\ast(Z_nX)\otimes_\qq H_\ast(Z_nX)\longrightarrow \qq\,,$$
for every family $X/S$. 
Via this pairing, $H^\ast$ is the dual   of $H_\ast$ as a $\qq$-vector space.  
We denote this pairing by $\langle e,\alpha\rangle$.

\begin{ex}\label{cop}
The coproduct of symmetric functions and the induced inhomogeneous enumerative theories are compatible with the product in $H_\ast$:
$$\langle e_{\Delta P},\alpha\otimes\beta\rangle=\langle e_{P},\alpha\cdot\beta\rangle\,.$$
Thus, we obtain a coalgebra morphism from the coalgebra $\Lambda$ of symmetric functions to the formal coalgebra $H^\ast$. If we denote by $\hat\Lambda$ the completion with respect to degree, $P\to e_P$ extends to morphism of formal coalgebras $\hat\Lambda\to  H^\ast$. 
\end{ex}

\begin{ex}
For proper $X\to \ast$ we have
$$\sum_{n=0}^\infty \langle c^k,[Z_nX]\rangle T^n=(e^T)^{\langle c(T_X)^k,[X]\rangle }$$
\end{ex}

\paragraph{Product.} As $H^\ast$ is the dual of $H_\ast$, the coproduct on $H_\ast$ induces a product on $H^\ast$.  Explicitly, if $c$ and $e$ are enumerative theories, their product $c\cdot e$ is defined by 
$$\langle c\cdot e,\alpha\rangle=\langle c\otimes e,\Delta(\alpha)\rangle\,.$$
Another way to say this is
$$c\cdot e=+_!\Delta^\ast(c\times e)\,.$$
The multiplication is graded with respect  to both degrees.  It is graded commutative with respect to cohomological degree.  Thus $H^\ast$ is a bigraded $\qq$-algebra. 
Cohomology is not a Hopf algebra, as it lacks the coproduct, although there is a formal coproduct.

\begin{ex}\label{ex:ck}
We have
$$\langle c^k,q_{N}\rangle=\prod_{\substack{i=1,\ldots,d\\j=1,\ldots,n}}\binom{k}{n_{ij}}\,,$$
so
$$\langle c^k,q_{\vec n}\rangle=\prod_{i=1,\ldots,d}\binom{k}{n_i}\,.$$

We also have
$$\langle e^k,q_N\rangle=\prod_{\substack{i=1,\ldots,d\\j=1,\ldots,n}}\delta_{kn_{ij}}
\,,$$
and
$$\langle e^k,q_{\vec n}\rangle=\prod_{i=1,\ldots,d}\delta_{kn_{i}}
\,,$$
\end{ex}

\subsection{Multiplicative theories} 

\begin{defn}
Multiplicative theories are those theories $e$, which satisfy
$$\langle e,\alpha\cdot\beta\rangle=\langle e,\alpha\rangle\langle e,\beta\rangle\,,$$
for elements $\alpha,\beta\in H_\ast$. Equivalently, for any two families $X\to S$ and $Y\to T$ we have
$$e((X_T\amalg Y_S)/(S\times T))=e(X/S)\times e(Y/T)\,.$$
\end{defn}

\begin{rmk}
In terms of the formal coproduct on $H^\ast$ this is equivalent to 
$$\Delta(e)=e\hat\otimes e\,,$$
or that $e$ is group-like, with respect to the formal Hopf algebra structure on   theories. 
\end{rmk}

 Multiplicative theories are never homogeneous (except the trivial one). The degree $0$ part of any multiplicative theory is~1.  
 
 Multiplicative theories are determined by their values on primitive homology classes. 

\begin{rmk}
Multiplicative theories form a commutative group under multiplication of theories. This follows from the fact that the comultiplication on $H_\ast$ is multiplicative. 
\end{rmk}

\begin{ex}
The theories $c^k$ and $e^k$ are multiplicative. 
\end{ex}
\begin{ex}
Let $P$  be a multiplicative generating function, or a group-like element of $\hat\Lambda$, the formal completion of the Hopf algebra of symmetric functions. This means that 
$$P(x_1,x_2,\ldots)=\prod_{i=1}^\infty P(x_i)\,,$$
for a power series $P(x)=1+t_1 x+t_2 x^2+\ldots$. 
Equivalently, 
$$P=\exp \sum_{n\geq 1}{a_n p_n}\,,$$
where the $p_n$ are the power sum symmetric functions (which are  primitive elements in the Hopf algebra of symmetric functions). Substituting Chern roots for the variables $x_i$, the induced characteristic class is multiplicative, i.e., $P(E_1\oplus E_2)=P(E_1)P(E_2)$. 

Multiplicative characteristic classes $P$ define  multiplicative enumerative theories by substituting the Chern roots of $T_{ZX/S}$ into $P$.  We denote the multiplicative theory defined by $P$ by $e_P$, see \ref{cherntang} and \ref{cop}.
$$e_P(X/S)=P(T_{ZX/S})\,.$$
The theory $c^k$ is defined by $P(x)=(1+x)^k$.
\end{ex}

\subsection{Primitive theories}

A theory $e\in H^\ast$ is {\bf primitive} if for any two $\alpha,\beta\in H_\ast$
$$\langle e,\alpha\cdot\beta\rangle=
\epsilon(\alpha)\langle e,\beta\rangle+
\epsilon(\beta) \langle e,\alpha\rangle\,,$$
where $\epsilon:H_\ast\to \qq$ is the counit (\ref{eq:counit}).

Another way to say this is that for  any two families $X\to S$, and $Y\to T$, we have
$$e((X_T\amalg Y_S)/(S\times T))=1\times e(Y/T)+e(X/S)\times 1\,.$$

\begin{rmk}
We can think of primitive theories as connected theories. 
\end{rmk}

Primitive theories are homogeneous: a theory is primitive if and only if its homogeneous components are primitive. 

\begin{prop}
There is a bijection between primitive and multiplicative theories given by exponential and logarithm.  For a primitive theory $a=\sum_{n>0}a_n$, we obtain a multiplicative theory
$$\exp(a)=\sum_{n\geq0} \frac{1}{n!} a^n$$
And for a multiplicative theory $e=1+\sum_{n\geq1}e_n$ we obtain a primitive theory by
$$\log(e)=\log(1+\sum_{n\geq1}e_n)=\sum_{k\geq1}\frac{(-1)^{k+1}}{k} \Big(\sum_{n\geq1}e_n\Big)^k$$
Under this correspondence, addition of primitive theories corresponds to multiplication of multiplicative theories. 
\end{prop}
\begin{pf}
This follows from the fact that $H^\ast$ is complete with respect to the grading by cycle degree.
\end{pf}

\begin{ex}
Every degree 1 theory is primitive. In fact, if $e$ is of degree 1, then
\begin{align*}
\langle e,\alpha\cdot\beta\rangle=&
\langle e,\alpha_0\beta_1\rangle+\langle e,\alpha_1 \beta_0\rangle\\
=&\epsilon(\alpha_0)\langle e,\beta_1\rangle+\epsilon(\beta_0)\langle e,\alpha_1\rangle\\
=&\epsilon(\alpha)\langle e,\beta\rangle+\epsilon(\beta)\langle e,\alpha\rangle
\end{align*}
essentially by the fact that $H_\ast$ is connected.

Thus, for any degree 1 theory $c$, the exponential $\exp c$ is a multiplicative theory.  
\end{ex}

\subsection{Applications}

We will write $[Z_nX]$ explicitly as a polynomial in the generators $q_{\vec n}$. 

\begin{thm}
Let $X\to\ast$ be a proper $d$-fold.
In $H_\ast$, we have $[ZX]=e^{[X]}$, or more accurately,
$$\sum_n [Z_nX]T^n=e^{ [X]T}\,.$$
Moreover, $[X]$ is a $\qq$-linear combination of $q_\lambda$, where $\lambda\vdash d$
$$[X]=\sum_{\lambda\vdash d}a_\lambda q_{\lambda}\,.$$
The $a_\lambda$ are universal polynomials in the Chern numbers of $X$.  We have
$$[X]=\sum_{\lambda\vdash d} \langle m_\lambda,[X]\rangle q_\lambda\,.$$
Here $m_\lambda$ is the monomial symmetric function associated to the partition 
$\lambda\vdash d$ evaluated at the Chern roots of $T_X$.
\end{thm}
\begin{pf}
Using the cover $X^n\to Z_nX=X/\Sigma_n$, which is finite \'etale of degee $n!$, we see that $[Z_nX]=\frac{1}{n!}[X^n]=\frac{1}{n!}[X]^n$, which proves the first formula. 

Since $[ZX]$ is group-like its logarithm is primitive, so $[X]$ is a linear combination of the primitives in $H_\ast$, which proves the second formula.  

To determine the values of the $a_\lambda$, let $m_\lambda$ be the multinomial symmetric function associated to the partition $\lambda\vdash d$. It gives rise to an enumerative theory $m_\lambda(T_Z)$.  Pair the second equation of the theorem with this theory to get
$$\langle m_\lambda(T_X),[X]\rangle=\sum_{\mu\vdash d}a_\mu \langle m_\lambda,q_\mu\rangle$$
We had seen that $\langle m_\lambda,q_\mu\rangle=\delta_{\lambda\mu}$ in the proof of Proposition~\ref{runs:over}.
\end{pf}

\begin{cor}
Let $c$ be a multiplicative theory, then 
$$\sum_n\langle c,[Z_nX]\rangle T^n=e^{\langle c,[X]\rangle T}\,.$$
For example, in the $d=1$ case,
$$\sum_n\langle c,[Z_nX]\rangle T^n=\Big(e^{\langle c,q_1\rangle T}\Big)^{\chi(X)}\,,$$
or in the $d=2$ case (with a slight abuse of notation for display purposes),
$$\sum_n\langle c,[Z_nX]\rangle T^n=\Big(e^{\langle c,q_{20}\rangle T}\Big)^{2\ch_2(X)}
                                    \Big(e^{\langle c,q_{11}\rangle T}\Big)^{\chi(X)}\,.$$
So this reduces the computation of all $\langle c,[Z_nX]\rangle$ to the computation of the two numbers $\langle c,q_{\smat{2\\0}}\rangle$ and $\langle c,q_{\smat{1\\1}}\rangle$.  This is  one computation for the bundle $\O\oplus\O(1)$ on $\pp^2$, and another for the bundle $\O(1,0)\oplus\O(0,1)$ on $\pp^1\times \pp^1$.                                     
\end{cor}

\begin{cor}
The vertical fundamental classes $[ZX]$, for $X\to\ast$ are contained in the subalgebra 
$$\qq[q_\lambda]_{\lambda\vdash d}\,.$$
\end{cor}

\section{Points on varieties: separated case}

A {\bf separated family }is a smooth {\em separated }morphism of Deligne-Mumford stacks $X\to S$.  The notion of family of cycles is the same as in Section~\ref{non-sep}, namely maps $D\to X$, such that $D\to S$ is finite \'etale and representable.

As we restrict to separated families, the moduli stack $Z_nX=X_S^n/\Sigma_n$ has finite stabilizers, and hence has a good moduli space $\ol Z_nX\to S$, whose formation commutes with base change.  This means that a cartesian diagram
$$\begin{tikzcd}[row sep=2ex]
X'\ar[d]\ar[r] & X\ar[d]\\
S'\ar[r] & S\end{tikzcd}$$
of families induces a cartesian diagram of relative cycle spaces
$$\begin{tikzcd}[row sep=2ex]
\ol Z_nX'\ar[r]\ar[d] & \ol Z_n X\ar[d]\\
S'\ar[r] & S\end{tikzcd}$$
This fact follows from the universality of moduli spaces in characteristic zero.  Of course, $\ol Z_nX\to S$ is nothing but the (relative) $n$-th symmetric power of $X\to S$. 

Let us remark that we work with relative moduli spaces, meaning that the morphism $Z_nX\to\ol Z_nX$ becomes the good moduli space whenever it is pulled back to a scheme $S'\to S$.  The stack $\ol Z_nX$ is representable over $S$.

We can interpret sections of the coarse space  $\ol Z_nX\to S$ as families of $0$-cycles in $X\to S$. 

\begin{defn}
We define $H_\ast^\sep$ to be the colimit of the $H_\ast(ZX)$ over the category of separated families $X\to S$. 
$$H_\ast^\sep=\injectlim_{\text{$X/S$ separated}}H_\ast(ZX)\,.$$
\end{defn}

As the category of separated families is a subcategory of the category of families, we have a canonical homomorphism 
$$H_\ast^\sep\longrightarrow H_\ast\,.$$

\begin{rmk}\label{rmk:shreak}
As $H_\ast(Z_nX)=H_\ast(\ol Z_n X)$, we have
$$H_\ast^\sep=\injectlim_{\text{$X/S$ separated}} H_\ast(\ol ZX)\,,$$
and so we can think of $H_\ast^\sep$ also as the Pardon algebra of $0$-cycles, rather than finite maps. 

As $\ol Z_nX$ is not generally smooth over $S$, one might think that for the Pardon homology based on $\ol Z$ one should use bivariant homology, which is 
$$H_\ast(\ol ZX\to S)=a_!\pi_!\pi^\ast a^!\qq\,,$$
with $a:S\to \ast$ and $\pi:\ol ZX\to S$ the structure maps. But, denoting the structure map by $\psi:\ol X\to X$, we have 
$F\longiso\psi_\ast\psi^\ast F$, for all constructible sheaves of $\qq$-vector spaces $F$ on $\ol ZX$.
In particular, $a_!\pi_!\pi^\ast a^!\qq\longiso a_!\pi_!\psi_!\psi^\ast\pi^\ast a^!\qq$, as $\psi$ is proper. Hence
$$H_\ast(ZX\to S)=H_\ast(\ol ZX\to S)\,.$$
As $\pi\psi:ZX\to S$ is smooth, of relative dimension $dn$, we have $(\pi\psi)^!=(\pi\psi)^\ast[2nd]$, and so
$$H_\ast(ZX\to S)=H_{\ast-2nd}(ZX)\,.$$
Thus, the natural degree on $H_\ast^\sep$ is defined by assigning the degree $i-2dn$ to elements of  $H_i^\sep(Z_nS)$.  
$$H_\ast^\sep=\injectlim_{\text{$X/S$ separated}} H_{\ast+2nd}(\ol ZX\to S)\,.$$
\end{rmk}

\begin{prop} 
The homomorphism $H_\ast^\sep\to H_\ast$ is a homomorphism of Hopf algebras. The Hopf algebra $H_\ast^\sep$ is graded, graded commutative and cocommutative, and hence also equal to the symmetric algebra on its subspace of primitives. 
\end{prop}
\begin{pf}
All the constructions that go into defining the product and the coproduct, and into proving their properties, stay within the category of separated families, if they start out there. This proposition follows directly. 

More interesting may be the fact that one can define product and coproduct in $H_\ast^\sep$ using constructions involving the moduli spaces $\ol Z$, instead of $Z$. For example, given separated families $X\to S$ and $Y\to T$, and $\alpha\in H_\ast(\ol ZX)$ and $\beta\in H_\ast(\ol ZY)$, their product in $H_\ast^\sep$ is represented by $\alpha\times\beta\in H_\ast(\ol Z X\times \ol ZY)=H_\ast(\ol Z(X_T\amalg Y_S))$. 

The coarse moduli map of $ZX\times _SZX\to ZX$, which represents union of source, is $\ol ZX\times _S\ol ZX\to \ol ZX$, which represents addition of $0$-cycles. The coproduct defined in terms of addition of $0$-cycles is equal to the one defined by union of source.
\end{pf}

\begin{defn}
Separated cohomology is defined as
$$H^\ast_\sep=\projectlim_{\text{$X/S$ separated}} H^\ast(ZX)\,.$$
Similar remarks as for homology apply:
$$H^\ast_\sep=\projectlim_{\text{$X/S$ separated}} H^\ast(\ol ZX)\,.$$
\end{defn}

There is a natural map of graded algebras
$$H^\ast\longrightarrow H^\ast_\sep\,.$$

We have pairings
\begin{equation}\label{pairing}
\begin{tikzcd}
H^\ast\otimes H_\ast^\sep\ar[r]\ar[d]\ar[dr] & H^\ast_\sep\otimes H_\ast^\sep\ar[d]\\
H^\ast\otimes H_\ast\ar[r]&\qq\rlap{\,.}\end{tikzcd}\end{equation}

\begin{ex}
Consider the case $d=1$.  In this case, the coarse space $\ol ZX\to S$ is smooth, for every separated family $X\to S$. We can therefore consider the enumerative theories 
$c^k(T_{\ol ZX/S})\in H^\ast(\ol ZX)=H^\ast(ZX)$ and $e^k(T_{\ol ZX/S})$, as analogues of $c^k$ and $e^k$ from Example~\ref{ex:pow}.  Let us denote these separated theories $c^k(T_{\ol Z})$ and $e^k(T_{\ol Z})$.

For a proper family $X\to \ast$ we also have the fundamental classes $[Z_nX]\in H_\ast(Z_nX)$ and $[\ol Z_n X]\in H_\ast(\ol Z_nX)$, which become equal under the identification $H_\ast(Z_nX)=H_\ast(\ol Z_nX)$. 

We have
$$\sum_{n=0}^\infty \langle e(T_{\ol Z}),[Z_nX]\rangle T^n=\sum_{n=0}^\infty \langle e(T_{\ol Z_nX}),[\ol Z_nX]\rangle T^n=\sum_{n=0}^\infty \chi(\ol Z_nX)T^n=\Big(\frac{1}{1-T}\Big)^{\chi( X)}\,,$$
and
$$\sum_{n=0}^\infty \langle e(T_{Z}),[Z_nX]\rangle T^n=\sum_{n=0}^\infty \langle e(T_{Z_nX}),[Z_nX]\rangle T^n=\sum_{n=0}^\infty \chi(Z_nX)T^n=\big(e^T\big)^{\chi( X)}\,.$$
\end{ex}
 
\begin{rmk}
It is impossible to lift or  extend the theory $e(T_{\ol Z})$ from $H^\ast_\sep$ to $H^\ast$.  We already remarked that the `obvious'  definition does not extend (Example~\ref{is:not}). 
Suppose, to the contrary, that such a lift existed, call it $e\in H^\ast$.  We use 
tautological classes $q_{\smat{\vec n\\\vec m}}$ introduced below. (Here $\vec n$ is the augmentation row, row zero, of the matrix $\smat{\vec n\\\vec m}$.) We have elements in the modules displayed in (\ref{pairing}):
$$\begin{tikzcd}
e\otimes \Big(q_{\smat{2\\2}}-\sfrac{1}{2}q_{\smat{1\\1}}^2-q_{\smat{1\\0}}q_{\smat{1\\2}}\Big)
\ar[r,maps to]
\ar[d,maps to]
 &e(T_{\ol Z})\otimes \Big(q_{\smat{2\\2}}-\sfrac{1}{2}q_{\smat{1\\1}}^2-q_{\smat{1\\0}}q_{\smat{1\\2}}\Big)\ar[d,maps to]\\
e\otimes\Big( \sfrac{1}{2}(q_1^2+2q_0q_1)-\sfrac{1}{2}q_1^2-q_0q_2\Big)=0
& 1-\sfrac{1}{2}-0\,,
\end{tikzcd}$$
which contradicts the commutativity of said diagram. We computed the vertical maps using Remark~\ref{3:15} and Remark~\ref{3:16}. 
\end{rmk}

The basic theory of vertical and horizontal classes applies to $H_\ast^\sep$ as well. The notions of primitive and multiplicative enumerative theories make sense in $H^\ast_\sep$, too. 

\begin{ex}\label{hilbert}
Consider the case $d=3$ of threefolds. For any separated family $X\to S$, proper or not, there is a relative Hilbert scheme $\Hilb^nX\to S$.  It comes with a virtual fundamental class $e^\vir(\Hilb^nX)\in H_0^{BM}(\Hilb^nX\to S)$. The Hilbert-Chow morphism $f:\Hilb^nX\to \ol Z_nX$ is proper, so we can push $e^\vir(\Hilb^nX)$ to $H_0^{BM}(\ol Z_n X\to S)=H^{6n}(\ol Z_nX)$. We obtain
$$f_\ast e^\vir(\Hilb^n)\in H^{6n}_\sep\,,$$
because  the virtual fundamental class on $\Hilb^n$ is natural with respect to arbitrary pullback of separated families, and Borel-Moore homology proper pushforward commutes with cohomology pullback. We denote the formal sum over all $n$ by $f_\ast e^\vir(\Hilb)\in H^\ast_\sep$. 
In fact $f_\ast e^\vir(\Hilb )$ is multiplicative.  This essentially follows from the fact that $\Hilb(X_T\amalg Y_S)=\Hilb(X)\times \Hilb(Y)$, and this equality is compatible with the Hilbert-Chow morphism. For proper $X\to \ast$, the value of
$$\langle f_\ast e^\vir(\Hilb),[\ol Z_nX]\rangle$$
is the Donaldson-Thomas virtual count of $\Hilb^n X$. 
 \end{ex}

\begin{ex}\label{inertia}
For a separated family $X\to S$, we have the moduli stack $Z_nX=X_S^n/\Sigma_n$, and its inertia stack
\begin{equation}\label{inert:def}
I_{Z_nX}=\Big(\coprod_{\sigma\in \Sigma_n} (X_S^n)^\sigma\Big)/\Sigma_n\,.
\end{equation}
It comes with a regular local immersion $i:I_{Z_nX}\to Z_nX$ which is proper, and hence admits cohomology pushforward.  Thus, characteristic classes on the inertia stack give rise to classes in Pardon cohomology $H^\ast_\sep$. In particular, any symmetric function $P(\gamma)$ gives rise to an associated {\bf inertial theory} $\ul P$ by 
$$\ul P(X/S)= P( T_{IZX})\,.$$
So we evaluate the symmetric function at the Chern roots of the Tangent bundle of the inertia stack.  If $P$ is multiplicative, then $\ul P$ is multiplicative. 

When $X\to S$ is not representable, calling (\ref{inert:def}) the inertia stack is a heavy abuse of language, as the inertia of $X$ itself will interfere in the inertia of $Z_nX$.  We always use (\ref{inert:def}) as the definition of $IZX$, even when $X\to S$ is stacky.\comment{this is my attempt to make this go through for stacks as well}
\end{ex}

\subsection{Tautological subalgebra}

Separated Pardon homology is potentially smaller, because there are fewer families, and hence fewer generators.  But as there are fewer families, there are also fewer relations, which makes separated Pardon homology potentially larger. Here we will show that as far as the tautological subalgebra is concerned, separated homology is, indeed, larger.  We will need more tautological generators.  We will introduce them here. 

As in Section~\ref{taut}, we pick a $d\times n$ matrix of non-negative integers $N$, but now we pick, in addition, an $n$-vector and insert it as the row number $0$ on top of $N$.  The augmented matrix we denote by $\ul N$, its augmentation row by $N_0$. The entries of $N_0$ are non-negative integers, as well.

We define a horizontal class associated to $\ul N$.  The parameter space is 
$$S=\pp=\prod_{\substack{i=1,\ldots,d\\j=1\ldots,n}}\pp^{n_{ij}}$$
as before. The rank $d$ bundles $E_j$, for $j=1,\ldots,n$, are the same as before as well, i.e.,  the sum of tautological bundles coming from all $\pp^{n_{ij}}$ corresponding to the $j$-th column of $N$. This gives rise to the same family of $d$-folds over $S$
$$X=E_1\amalg\ldots\amalg E_n$$
as before. 

The difference lies now in the cycle family, namely we define 
$$D=\coprod_{j=1,\ldots,n}D_j=\coprod_{j=1,\ldots,n} \pp\times \ul N_{0j}\,.$$
This is a disjoint union of copies of $\pp$; for each $j$ the number of copies is given by the $j$-th entry of the augmentation row $N_0$.  We map $D$ to $X$ by sending $D_j$ to $E_j$ as the $N_{0j}$-fold zero section.  This defines a finite map (or $0$-cycle) of degree $|N_0|=N_{01}+\ldots+N_{0n}$ in $X$.
$$\begin{tikzcd}
\coprod_{j=1,\ldots,n} D_j\ar[r]\ar[dr] & \coprod_{j=1,\ldots,n} E_j\ar[d]\\
& \pp\end{tikzcd}$$
We now define the horizontal class
$$q_{\ul N}=\frac{1}{N_{01}!\ldots N_{0n}!}\eta(D/X/S,[S])\in H_{2|N|}^\sep\,.$$
With this new notation, the previous $q_N$ is the new $q_{\ul N}$ with the augmentation row consisting of $1$s entirely.

Again, because they will turn out to be algebra generators, let us emphasize the case $n=1$.  These tautological classes are associated to $(d+1)$-vectors $\ul {\vec n}$, they are denoted $q_{\ul{\vec n}}$, and they are defined by the vector bundle $E=\O(1,0,\ldots,0)+\ldots+\O(0,\ldots,0,1)$ over $\pp^{n_1}\times\ldots\times \pp^{n_d}$ and the $n_0$-fold zero section of $E$. We call these also the {\bf equivariant multiple points}. The degree of an equivariant multiple point is $n_0$. Note that it is normalized with the factor $1/n_0!$.

\paragraph{N.B.} Number of rows: dimension of families, number of columns: number of connected components (i.e. disjoint union of vector bundles), augmentation row: multiplicity of section/cycle.

\begin{rmk}
If $\ul N'$ is obtained from $\ul N$ by permuting the columns, we have
$$q_{\ul N'}=q_{\ul N}\,.$$
If $\ul N'$ is obtained from $\ul N$ by permuting the rows other than the augmentation row, 
$$q_{\ul N'}=q_{\ul N}\,.$$
We also have
$$q_{\ul N}=0\,,$$
if the augmentation row is the zero vector of size $n$, unless $N$ is zero as well, in which case $q_{\ul N}=1$. 
Thus, if we remove trailing zeros, then for $q_{\ul N}$ to be non-zero, the zeroth row has to consist of positive integers.
\end{rmk}

\begin{prop}
For the product of   tautological classes $q_{\ul N}$ and $q_{\ul M}$ we have
$$q_{\ul N}\cdot q_{\ul M}=q_{\ul N\amalg \ul M}\,.$$
For the coproduct of the tautological class $q_{\ul N}$ we have 
$$\Delta(q_{\ul N})=\sum_{\ul N=\ul P+\ul Q} q_{\ul P}\otimes q_{\ul Q}\,.$$
\end{prop}
\begin{pf}
The claim for the product is straightforward.  Let us prove the formula for the coproduct. By Proposition~\ref{coprodh}, we have
\begin{align*}
\Delta(q_{\ul N})=&\prod_j \frac{1}{N_{0j}!}\Delta\eta(\amalg D_j/\amalg_j E_j/\pp^N,[\pp^N])\\
=&\prod_j\frac{1}{N_{0j}!}\sum_{\substack{N_0=N_0'+N_0''\\N=N'+N''}}\prod_j \binom{N_{0j}}{N_{0j}'}
\eta(\amalg D'_j/\amalg_j E_j/\pp^N,[\pp^{N'}])
\otimes\eta(\amalg D''_j/\amalg_j E_j/\pp^N,[\pp^{N''}])\\
=&\sum_{\ul N=\ul N+\ul N'}\prod_j \frac{1}{N_{0j}'!}
\eta(\amalg D'_j/\amalg_j E_j/\pp^{N'},[\pp^{N'}])
\otimes\prod_j \frac{1}{N_{0j}''}\eta(\amalg D''_j/\amalg_j E_j/\pp^{N''},[\pp^{N''}])\\
=&\sum_{\ul N=\ul N'+\ul N''} q_{\ul N'}\otimes q_{\ul N''}\,.
\end{align*}
Here by $D'_j$ and $D''_j$ we denote the $N_{0j}'$ and $N_{0j}''$-fold $0$-sections of $E_j$. The binomial factors account for the number of partitions of the indexing set which have the same cardinality. We also use the Pardon relations when replacing the parameter space $\pp^N$ by $\pp^{N'}$ and $\pp^{N''}$, respectively. 
\end{pf}

\begin{cor}
The $\qq$-vector subspace of $H_\ast^\sep$ spanned by all $q_{\ul N}$  is a Hopf subalgebra.  We call it the {\bf tautological }Pardon algebra for $0$-cycles in $d$-folds. Multiplicative generators of the tautological algebra are the $q_{\ul{\vec n}}$. 
\end{cor}

\begin{cor}
For every $n$, the  formal sum $\sum_{\ul M} q_{\ul M}$ (sum over all $d\times n$-matrices) is group-like.
\end{cor}

\begin{cor}
Define, for every $\ul{\vec n}\in \nn\times \nno^{d}$ 
$$p_{\ul{\vec n}}=
\sum_{k=1}^\infty \frac{(-1)^{k+1}}{k}\sum_{\substack{\ul{\vec n}=\ul{\vec n}_1+\ldots+\ul{\vec n}_k\\
\ul{\vec n}_i\not=0}}\,\, \prod_{i=1}^k q_{\ul{\vec n}_i}\,.$$
These classes are primitive in $H_\ast^\sep$ and generate the tautological Hopf algebra multiplicatively.  They therefore span the space of primitives in $H_\ast^{\sep,\taut}$.  To get a basis of the space of primitives, restrict to vectors $\ul{\vec n}=(n,\vec m)$ where $\vec m$ is sorted by size of its entries, and $\vec m\geq n-1$.

As a $\qq$-algebra, we have 
$$H_\ast^{\sep,\taut}=\qq[q_{n,\vec m}]_{{\vec m}\geq n-1}\,.$$
 As a Hopf algebra, we have 
 $$H_\ast^{\sep,\taut}=\qq[p_{n,\vec m}]_{\vec m\geq n-1}$$ (by which we mean that the Hopf algebra is the symmetric Hopf algebra on the given generators).

\end{cor}
\begin{pf}
By Remark~\ref{redm}, the $q_{n, \vec m}$ with $\vec m\geq n-1$ generate $H_\ast^\taut$ as a $\qq$-algebra.  It is not hard to see that the same is true for the $p_{n,\vec m}$ with $\vec m\geq n-1$. 
We will prove that the $p_{n,\vec m}$ with $\vec m\geq n-1$ are algebraically independent over $\qq$. From this, it will follow that the $q_{n\vec m}$, with $\vec m\geq n-1$ are algebraically independent as well.

  We let $P(\gamma)=1+\sum_{j=1}^\infty t_j \gamma^j$ be the generic multiplicative characteristic class. Then we have (see Theorem~\ref{iner})
\begin{align*}
\sum_{\vec m}\langle \ul P,p_{n\vec m}\rangle  U^{\vec m}=&\frac{1}{n}
  U^{n-1}P(U_1)\ldots P(U_d)\\
=& \frac{1}{n} U^{n-1} \sum_{\vec m} t_{m_1}\ldots t_{m_d} U_1^{m_1}\ldots U_d^{m_d}\\
=&\frac{1}{n}\sum_{m}  t_{\vec m+1-n} U^{\vec m}\,.
\end{align*}
This implies that
$$\langle \ul P,p_{n\vec m}\rangle=\frac{1}{n}t_{\vec m+1-n}\,.$$
Now suppose $\sum_{\vec m} a_{\vec m}p_{n,\vec m}=0.$ This implies
$$\sum_{m_1,\ldots,m_d} a_{m_1,\ldots,m_d} t_{m_1+1-n}\ldots t_{m_d+1-n}=0\,,$$
or,
$$\sum_{m_1,\ldots,m_d} a_{m_1+n-1,\ldots,m_d+n-1} t_{m_1}\ldots t_{m_d}=0\,.$$
This gives, for every $d$-tuple of non-negative integers, $(m_1,\ldots,m_d)$ that
$$\sum_{\sigma\in \Sigma_d} a_{m_{\sigma(1)},\ldots,m_{\sigma(d)}}=0\,.$$
This is sufficient for proving our claim.
\end{pf}

\begin{rmk}\label{3:15}
The homomorphism of Hopf algebras $H_\ast^{\sep,\taut}\to H_\ast^{\taut}$ if given on algebra generators by
\begin{align*}
\qq[q_{n\vec n}]&\longrightarrow \qq[q_{\vec m}]\\
q_{n\vec n}&\longmapsto 
\frac{1}{n!}\sum_{\vec n=\vec m_1+\ldots+\vec m_n}\prod_{i=1}^nq_{\vec m_i}\,.
\end{align*}
and on Hopf algebra generators by
\begin{align*}
\qq[p_{n\vec n}]&\longrightarrow \qq[q_{\vec m}]\\
p_{n\vec n}&\longmapsto \begin{cases}
q_{\vec n} & \text{if $n=1$}\\ 0 & \text{else}\end{cases} 
\end{align*}
We have written the augmentation row (which is just a single entry) in front, rather than above, to save vertical space. 
\end{rmk}

\begin{rmk}\label{3:16}
Consider the case of curves, $d=1$. The separated enumerative theories $c^k(T_{\ol Z})$ and $e^k(T_{\ol Z})$ evaluate on the equivariant multiple points $q_{nm}$ as follows:
$$\sum_{m=0}^\infty\langle c^k(T_{\ol Z}),q_{nm}\rangle T^m=\frac{1}{n!}\prod_{i=1}^n (1+iT)^k   $$
$$\sum_{m=0}^\infty\langle e^k(T_{\ol Z}),q_{nm}\rangle T^m=(n!)^{k-1}T^{nk}\,.$$
To do these calculations, recall that $q_{nm}$ is associated to the degree $n$ map $(\pp^m)^{\amalg n}\stackrel{0}{\longrightarrow} \O_{\pp^m}(1)\to\pp^m$ and the fundamental class $[\pp^m]$. The corresponding relative cycle space is $\ol Z_n\O_{\pp^m}(1)\to\pp^m$, which is, in fact, $\O(1)\oplus\ldots\oplus \O(n)\to \pp^m$. (And the classifying section of the map is the zero section of the rank $n$ bundle $\O(1)\oplus\ldots\oplus \O(n)$.) Therefore,
\begin{align*}
\langle c^k(T_{\ol Z}),q_{nm}\rangle=&\frac{1}{n!}\langle c(\O(1)+\ldots+\O(n))^k,[\pp^m]\rangle\\
=&\frac{1}{n!}\langle(1+c_1)^k\ldots (1+nc_1)^k,[\pp^m]\rangle\,.
\end{align*}
\end{rmk}

\begin{ex}
We can lift the formulas from Example~\ref{ex:ck} to $H_\ast^\sep$. In fact
$$\langle c^k,q_{n\vec m}\rangle=\frac{1}{n!}\prod_{i=1}^d \binom{kn}{m_i}\,.$$
$$\sum_{n,m}\langle c^k,q_{nm}\rangle T^n U^m=(e^T)^{(1+U)^k}\,.$$
To prove these, consider the diagram
$$\begin{tikzcd}
E^{\oplus n}\ar[r,"f"]\ar[dr] &  E^{\oplus n}/\Sigma_n\ar[d,"\pi"]\\
& \pp\ar[ul,bend left,"0"]\end{tikzcd}$$
Here, $\pp=\pp^{m_1}\times\ldots\times \pp^{m_n}$ is the base of our family, $E=\O(1,0,\ldots,0)+\ldots+\O(0,\ldots,0,1)$ is the total space.  The relative moduli space is $Z_nE=E^{\oplus n}/\Sigma_n$. Our horizontal class is $q_{n\vec m}=\frac{1}{n!}f_\ast 0_\ast [\pp]$. So
\begin{align*}
n!\langle c^k,q_{n,\vec m}\rangle=&\langle c^k(T_{\pi}),f_\ast 0_\ast[\pp]\rangle\\
=&\langle f^\ast c^k(T_\pi),0_\ast[\pp]\rangle\\
=&\langle c^k(f^\ast\pi^\ast E^{\oplus n}),0_\ast[\pp]\rangle\\
=&\langle c^k(E^{\oplus n}),[\pp]\rangle\\
=&\langle (1+c_1)^{kn},[\pp^{m_1}]\rangle\ldots\langle (1+c_1)^{kn},[\pp^{m_n}]\rangle\\
=&\binom{kn}{m_1}\ldots\binom{kn}{m_n}\,.
\end{align*}
\end{ex}

\begin{thm}\label{ttsca}
The tautological subalgebra $H_\ast^{\sep,\taut}$ contains all horizontal, and hence also all vertical classes.
\end{thm}
\begin{pf}
All horizontal classes are vertical, and we can decompose $D\to X$ into sections of $X\to S$.

We cannot use the deformation to the normal cone of $D\to X$, unless $D\to X$ is a closed immersion, because we are not allowing non-separated families.   
Instead, we use a closely related deformation technique: double point degenerations.  For these, the base $S$ of a family $X\to S$ is deformed as well. Details in the ensuing section.
\end{pf}

\subsection{Double point degenerations}

\paragraph{Recall of the basic construction.}
Let $Z\to X$ be a regular closed immersion.  The associated {\bf double point degeneration }$M_ZX$ (see \cite{Fulton}) is the blow-up of $X\times \aaa^1$ along $Z$, where $Z$ is embedded in $X\times\aaa^1$ by using $0$ in the second component.
$$M_ZX=\Bl_Z (X\times\aaa^1)=\Bl_{(Z\times 0)}(X\times \aaa^1)\,.$$
The double point degeneration comes with a closed embedding $Z\times\aaa^1\to \Bl_Z(X\times\aaa^1)$, and is fibered over $\aaa^1$; the composition $Z\times \aaa^1\to\aaa^1$ is the projection:
$$\begin{tikzcd}
Z\ar[d]\ar[r] & Z\times \aaa^1\ar[d] & Z\ar[l]\ar[d]\\
\pp(N_{Z/X}+1)\amalg_{\pp(N_{Z/X})}\Bl_Z X\ar[r]\ar[d] & \Bl_Z(X\times\aaa^1) \ar[d]& \ar[l] X\ar[d]\\
0\ar[r]& \aaa^1 & 1\ar[l]\rlap{\,.}\end{tikzcd}$$
The morphism $\Bl_Z(X\times\aaa^1)\to \aaa^1$ is flat. 
The fibre over $1\in\aaa^1$ (as well as all other points not equal to $0$) is the given embedding $Z\to X$. 

The fibre over $0$ is a Cartier divisor in $\Bl_Z(X\times \aaa^1)$, which is the sum of the two components $\pp(N_{Z/X}+1)$ and $\Bl_ZX$, intersecting along $\pp(N_{Z/X})$. (This is what the $\amalg$-notation is supposed to indicate.) The projective bundle $\pp(N_{Z/X})$ is embedded in $\pp(N_{Z/X}+1)$ as the hyperplane bundle at $\infty$, and in the blow-up $\Bl_ZX$ as the exceptional divisor. The embedding of $Z$ into the fibre over $0$ of $\Bl_Z(X\times\aaa^1)$ is as the zero section of the vector bundle $N_{Z/X}$, which is embedded in its projective completion $\pp(N_{Z/X}+1)$ in the usual way.

The blow-up $\Bl_ZX$ is embedded in $\Bl_{Z}(X\times \aaa^1)$ as the strict transform of $X\times 0$, the projective bundle $\pp(N_{Z/X}+1)$ is embedded in $\Bl_Z(X\times\aaa^1)$ as the exceptional divisor of the blow-up $\Bl_Z(X\times\aaa^1)$. 

\paragraph{Family version.}
We need a relative version  of this construction. For this, we consider a diagram
$$\begin{tikzcd}
Z\ar[r]\ar[dr] & X_W\ar[r]\ar[d]& X\ar[d]\\
        & W\ar[r]      & S\rlap{\,,}
\end{tikzcd}$$
where  $X\to S$ and $Z\to W$ are smooth, $W\to S$ is a regular closed immersion, and the square is cartesian. It follows that $Z\to X_W$ and $Z\to X$ are regular closed immersions as well.   Then we have an induced morphism of double point degenerations
$$\rho:Bl_Z(X\times\aaa^1)\setminus \Bl_Z X_W\longrightarrow \Bl_W(S\times \aaa^1)\,.$$

$$\begin{tikzcd}
\pp(N_{Z/X}+1)\setminus\pp(N_{Z/X_W})\amalg_{\pp(N_{Z/X})\setminus\pp(N_{Z/X_W})}\Bl_Z X\setminus \Bl_Z(X_W)\ar[r]\ar[d] & 
\Bl_Z(X\times\aaa^1)\setminus \Bl_Z(X_W) \ar[d,"\rho"]& \ar[l] X\ar[d]\\
\pp(N_{W/S}+1)\amalg_{\pp(N_{W/S})}\Bl_W S\ar[r]\ar[d] & \Bl_W(S\times\aaa^1) \ar[d,]& \ar[l] S\ar[d]\\
0\ar[r]& \aaa^1 & 1\ar[l]\rlap{\,.}\end{tikzcd}$$

\paragraph{Claim.}
The morphism $\rho$ is smooth. This follows from a version of the local criterion for flatness (see \cite{stacks-project}, Tag~00MD,  Lemma~10.99.16), which asserts that to prove that an $A$-module $M$ is flat, it suffices to find $f\in A$ which is a non-zero divisor on both $A$ and $M$, and then check that both $M_f$ is flat over $A_f$ and $M/fM$ is flat over $A/fA$. We can apply this criterion, because a blow-up is the disjoint union of the exceptional divisor, which is a Cartier divisor, and the open complement. Applying this criterion, it suffices to check flatness of $X\times\aaa^1\setminus X_W\to S\times\aaa^1\setminus W$ and $\pp(N_{Z/X}+1)\setminus\pp(N_{Z/X_W})\to \pp(N_{W/S}+1)$. This flatness will then follows from smoothness 
of both of these maps, which will then also prove the overall smoothness. The morphism on the exceptional divisors $\pp(N_{Z/X}+1)\setminus\pp(N_{Z/X_W})\to \pp(N_{W/S}+1)$, is the composition of the projection of projective bundles associated to the epimorphism of vector bundles $N_{Z/X}+\O_Z\to N_{W/S}|_Z+\O_Z$ and the morphism $\pp(N_{W/S}+1)|_Z\to \pp(N_{W/S}+1)$, which pulls back from the smooth morphism $Z\to W$. 

By this claim, $\rho$ is a family of $d$-folds. Note that none of these blow-up constructions produce non-separatedness, so $\rho$ is separated. 

\paragraph{Adding a map to the mix.}
We now bring in a finite map $D\to X\to S$, i.e., the composition $D\to S$ is finite \'etale representable. We furthermore assume that the induced morphism $Z\times_XD\to W\times_SD$ is an isomorphism. 
Then we have an expanded diagram
where all squares are cartesian, 
$$\begin{tikzcd}
D_Z\ar[r,"\sim"]\ar[d] & D_W\ar[r]\ar[d] & D\ar[d]\\
Z\ar[dr]\ar[r] & X_W\ar[r]\ar[d] & X\ar[d]\\
& W\ar[r] & S
\end{tikzcd}$$
We get   induced morphisms
\begin{equation}\label{hurrah}
\Bl_{D_Z}(D\times\aaa^1)\longrightarrow \Bl_Z(X\times\aaa^1)\setminus\Bl_Z(X_W)
\stackrel{\rho}{\longrightarrow} \Bl_W(S\times\aaa^1)\,.\end{equation}

$$\begin{tikzcd}
\pp(N_{D_Z/D}+1)
\amalg_{\pp(N_{D_Z/D})}
\Bl_{D_Z}(D)\ar[r]\ar[d] & 
\Bl_{D_Z}(D\times\aaa^1) \ar[d]& \ar[l] D\ar[d]\\
\pp(N_{Z/X}+1)\setminus\pp(N_{Z/X_W})
\amalg_{\pp(N_{Z/X})\setminus\pp(N_{Z/X_W})}
\Bl_Z X\setminus \Bl_Z(X_W)\ar[r]\ar[d] & 
\Bl_Z(X\times\aaa^1)\setminus \Bl_Z(X_W) \ar[d,"\rho"]& \ar[l] X\ar[d]\\
\pp(N_{W/S}+1)\amalg_{\pp(N_{W/S})}\Bl_W S\ar[r] & \Bl_W(S\times\aaa^1) & \ar[l] S 
\end{tikzcd}$$

Note that we do not need to remove anything from the source because the preimage of $\Bl_Z(X_W)$ under the morphism $\Bl_{D_Z}(D\times\aaa^1)\to \Bl_Z(X\times\aaa^1)$ is $\Bl_{D_Z}(D_W)$, which is empty.  By the above claim, the morphism $\Bl_{D_Z}(D\times\aaa^1)\to \Bl_W(S\times \aaa^1)$ is smooth of relative dimension zero, i.e., \'etale.  It is also proper and representable, hence finite representable. 

We conclude that (\ref{hurrah}) defines a finite map into a smooth family. It has the same degree as $D\to S$.

\begin{lem}\label{dumb}
For every homology class $\alpha\in H_\ast(S)$, there exist homology classes $\beta\in H_\ast(\Bl_WS)$  and $\gamma\in H_\ast(\pp(N_{W/S}+1))$, such that in $H_\ast(\Bl_W(S\times\aaa^1))$ we have 
$\alpha=\beta+\gamma$. 
\end{lem}
\begin{pf}
The double point degeneration $\Bl_W(S\times\aaa^1)$ contracts onto the central fibre, so it has the same homology as the central fibre. This defines a pushforward map from the generic fibre to the central fibre.  Then one uses that the homology of the central fibre is generated (additively), by classes coming from its two components. 
\end{pf}

\begin{cor}\label{3:17}
Any horizontal class $\eta(D/X/S,\alpha)$ defined in terms of  $D\to X\to S$ and a class $\alpha\in H_\ast(S)$, is, in $H_\ast^\sep$, equivalent to a linear combination of horizontal classes defined in terms of the two families 
$$\Bl_{D_Z}(D)\longrightarrow\Bl_Z(X)\setminus \Bl_Z(X_W)\longrightarrow \Bl_WS$$
 and 
 $$\pp(N_{D_Z/D}+1)\longrightarrow \pp(N_{Z/X}+1)\setminus\pp(N_{Z/X_W})\longrightarrow \pp(N_{W/S}+1)\,.$$
\end{cor}

\paragraph{Stratifications.}  

For a family of $d$-folds $X\to S$, we consider the following diagonal:
\begin{equation}\label{angst}
\coprod_{\ul n\epi\ul k}\mathring X^k_S\longrightarrow X^n_S\,.
\end{equation}
We can think of a point on the left hand side as a map $\ul n\epi\ul k\to X$, and a point on the right hand side as a map $\ul n\to X$, and the morphism on display eliminates the $\ul k$ in the middle. The notation $\mathring X_S^k$ denotes the open complement of the big diagonal. With this stipulation, (\ref{angst}) is a regular locally closed immersion. The images of all(\ref{angst}) as $1\leq k\leq n$ form a stratification of the target.

The symmetric group acts on the morphism (\ref{angst}), and we pass to the quotient:
$$Z_n^{(k)}X=\Big(\coprod_{\ul n\epi\ul k}\mathring X^k_S\Big)/\Sigma_n\longrightarrow X^n_S/\Sigma_n=Z_nX\,.$$
For every $k$ this is a regular locally closed immersion and the union over all $k$ is the stratification of $Z_nX$ by {\em number of values}. 

We also have the universal degree $n$ map
$$X^n_S\times_{\Sigma_n}\ul n\longrightarrow X^n_S/\Sigma_n\times_SX\,.$$
When we pull it back via $Z_n^{(k)}X\to Z_nX$, it factors through a degree $k$ map which is a closed immersion:
\begin{equation}\label{factor}
\Big(\coprod_{\ul n\epi\ul k}\mathring X^k_S\Big)\times_{\Sigma_n}\ul n
\longrightarrow \Big(\coprod_{\ul n\epi\ul k}\mathring X^k_S\Big)/\Sigma_n \times\ul k
\longrightarrow \Big(\coprod_{\ul n\epi\ul k}\mathring X^k_S\Big)/{\Sigma_n}\times X\,.
\end{equation}

If we now have a degree $n$ family of finite maps $D\to X\to S$, we get an induced classifying morphism $\Delta:S\to Z_nX$, the pullback of the stratification $(Z_n^{(k)})_k$ via this map $\Delta$ is the stratification of $S$ by number of values of $D\to X$. Let us use $S^{(k)}$ as notation of the strata. Note that although the universal strata $Z_n^{(k)}\to Z_nX$ are regularly immersed, this is not true in general for the strata $S^{(k)}\to S$. 

Via the classifying morphism $S\to Z_nX$, we also pull back the factorization of the universal degree $n$ map (\ref{factor}), to factor $D|_{S^{(k)}}$.  We get the following diagram
\begin{equation}\label{final:diag}\begin{tikzcd}
D\vert_{D^{(k)}}\ar[r,"\sim"]\ar[d] & D\vert_{S^{(k)}}\ar[r]\ar[d]\ar[dl] & D\ar[d]\\
D^{(k)}\ar[dr]\ar[r] & X\vert_{S^{(k)}}\ar[r]\ar[d] & X\ar[d]\\
& S^{(k)}\ar[r] & S
\end{tikzcd}\end{equation}

\begin{figure}[h]
\centering
\includegraphics[scale=0.2]{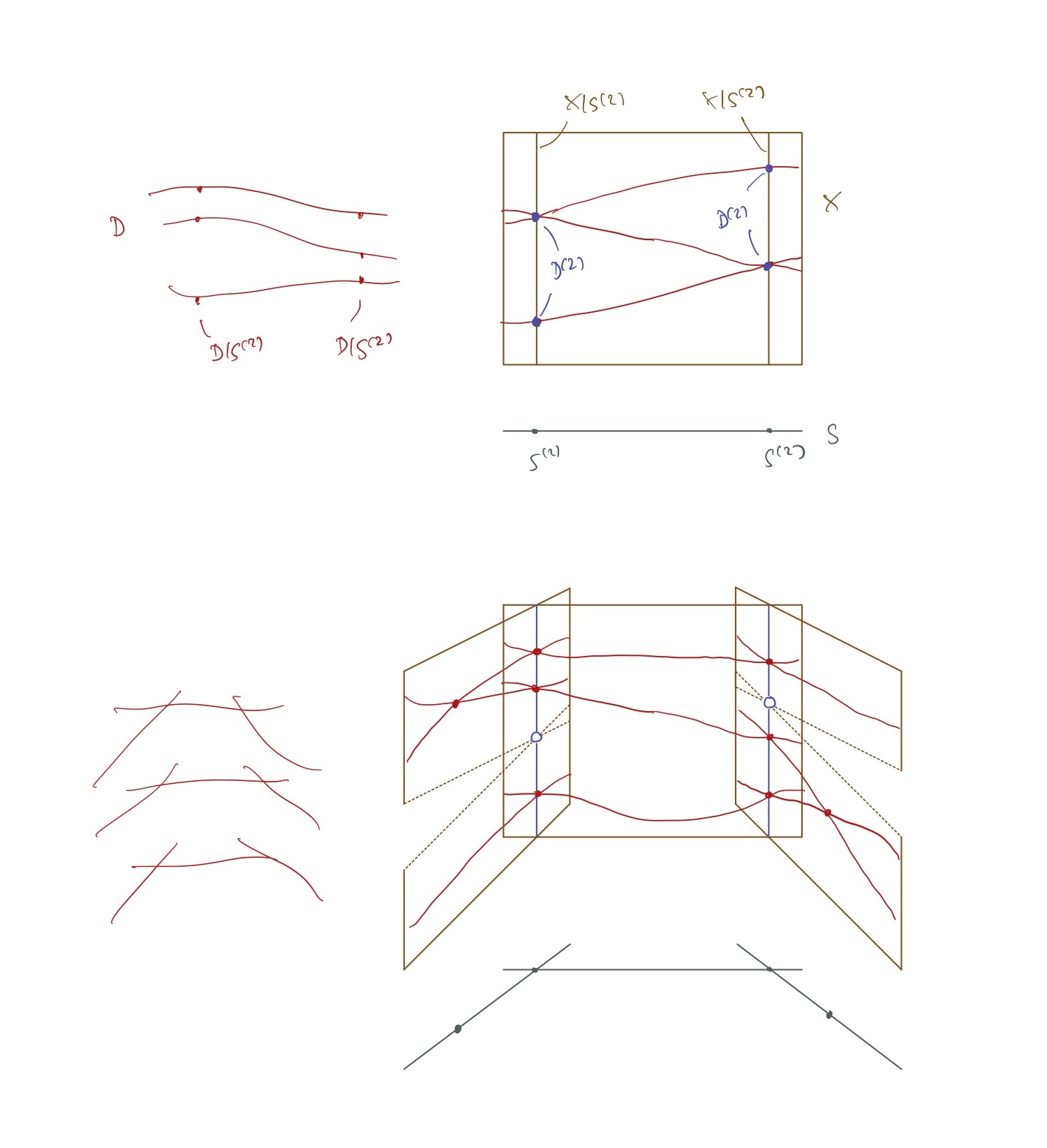}
\caption{Before and after double point degeneration}
\end{figure}

\paragraph{Using double point degenerations to simplify things.} 
We return to the proof of Theorem~\ref{ttsca}.
We consider a finite map of degree $n$  in a smooth family $D\to X\to S$. 

For the sake of an inductive argument, let us assume that $S^{(\ell)}=\varnothing$, for $\ell<k$.  This will imply that $S^{(k)}$ is a closed subscheme of $S$.  By passing to the universal situation, if necessary, we may assume that $S^{(k)}\to S$ is a regular closed immersion. The diagram (\ref{final:diag}) then satisfies all the conditions we imposed above in the lead-up to Corollary~\ref{3:17}. 

We may therefore employ Corollary~\ref{3:17}, and replace the family $D\to X\to S$ with the two families 
$$\Bl_{D|_{D^{(k)}}}(D)\longrightarrow\Bl_{D^{(k)}}(X)\setminus \Bl_{D^{(k)}}(X|_{S^{(k)}})\longrightarrow \Bl_{S^{(k)}}(S)$$
 and 
 $$\pp(N_{D|_{D^{(k)}}/D}+1)\longrightarrow \pp(N_{D^{(k)}/X}+1)\setminus
 \pp(N_{D^{(k)}/X|_{S^{(k)}}})
 \longrightarrow \pp(N_{S^{(k)}/S}+1)\,.$$

We   deduce that a horizontal class defined by $n$ sections with at least $k$ different values is a linear combination of horizontal classes of two types:
\begin{items}
\item horizontal classes defined by $n$ sections with at least $k+1$ distinct values,
\item horizontal classes for families consisting of disjoint unions of vector bundles.
\end{items}
By induction, we conclude that every horizontal class defined by $n$ sections is a linear combination of horizontal classes of two types:
\begin{items}
\item horizontal classes defined by $n$ sections that are pairwise disjoint,
\item horizontal classes for families consisting of disjoint unions of vector bundles.
\end{items}
Then, as in Section~\ref{non-sep}, we can deform to the normal cone to reduce the first type to the second.  This is now legal, because $S^{\amalg n}\to X$ is a closed immersion, and so the deformation to the normal cone is a separated family.  But we cannot avoid that the sections intersect in the second type of family. 

Now we can continue as in Section~\ref{non-sep} to conclude the proof. 

\begin{rmk}\label{redm}\comment{should carefully write down a proof that the $p_{n,\vec m}$ with $n-1\leq \vec m$ generate.  Also, do we want to conjecture that the $p_{n,\vec m}$ with $n-1\not\leq \vec m$ are zero?}
Consider one of the multiplicative generators $q_{n,\vec m}$.  If $n-1>\min( \vec m)$,  the $n$-fold zero section can be deformed to $n$ generic sections, which   have no common values. When performing the double point degeneration, these are never produced. Therefore, only $q_{n,\vec m}$ with $n-1\leq\min (\vec m)$ or $n-1\leq \vec m$ (by which we mean that $n-1\leq m_i$ for all $i=1,\ldots,d$) are used to generate the tautological $H_\ast$ multiplicatively.

To generate the tautological Pardon homology algebra, we may therefore restrict to $q_{n,\vec m}$ with $n-1\leq\min(\vec m)$, or $n-1\leq \vec m$ (by which we mean that $n-1\leq m_i$ for all $i=1,\ldots,d$. 
In particular, for fixed $\vec m$, only finitely many $n$ are needed. 
\end{rmk}

\subsection{Consequences I}

Here we deal only with the case $d=1$ of curves.

\begin{prop}
For a proper curve $X\to \ast$, we can express the vertical classes $[Z_nX]=[\ol Z_nX]\in H_\ast^{\sep,\taut}$ in terms of the primitive classes $p_{\smat{n\\m}}$, which we will write as $p_{nm}$, by the formula
$$\sum_n[Z_nX]T^n=\Big(\exp\sum_n p_{nn}T^n\Big)^{\chi(X)}\,.$$
Unfortunately, we were not able to find such a simple formula in terms of the tautological classes $q_{nm}$. 
\end{prop}
\begin{pf}
Let us denote by $\ol e=e(T_{\ol Z})$. The Euler class $e=e(T_{Z})$ is less useful here, $\langle e,p_{nn}\rangle=0$, unless $n=1$. 

We introduce formal variables $T$, which will track the degree, and $U$ which will track half the homological degree. 
We pair the equation 
$$\sum_{n,m}p_{nm}T^nU^m=\log \sum_{n,m}q_{nm}T^nU^m$$
 with $\ol e$ to get
\begin{align*}
\sum_{n,m}\langle \ol e,p_{nm}\rangle T^nU^m
=&\log \sum_{n,m}\langle \ol e,q_{nm}\rangle T^nU^m\\
=&\log \sum_{n,m}\delta_{nm} T^nU^m\\
=&\log\frac{1}{1-TU}\\
=&\sum_{n}\frac{1}{n}(TU)^n\,.
\end{align*}
 This works because $\ol e$ is multiplicative.
We conclude that $\langle \ol e,p_{nm}\rangle=0$, if $n\not=m$, and
$\langle \ol e,p_{nn}\rangle=\frac{1}{n}$. 

We consider a proper curve $X\to \ast$. Since  $\sum_{n} [Z_nX]T^n U^n$ is group-like, its logarithm is primitive, and therefore a linear combination of the $p_{n,m}T^nU^m$. 
$$\log\sum_{n}[Z_nX](TU)^n=\sum_{n,m} a_{nm} p_{nm}T^nU^m\,.$$
Projecting into total degree zero, we get the equation
$$\log\sum_{n}[Z_nX](TU)^n=\sum_{n} a_{nn} p_{nn}(TU)^n\,.$$
Let us pair this equation with $\ol e$:
$$\log\sum_{n}\langle \ol e,[Z_nX]\rangle(TU)^n=\sum_{n} a_{nn} \langle \ol e,p_{nn}\rangle(TU)^n\,,$$
or
$$\log \Big(\frac{1}{1-TU}\Big)^{\chi(X)}=\sum_{n} \frac{a_{nn}}{n}(TU)^n\,,$$
which implies that $a_{nn}=\chi(X)$, for all $n$. We conclude that
$$\log\sum_n[Z_nX](TU)^n=\chi(X)\sum_n p_{nn}(TU)^n\,,$$
or,
$$\sum_n[Z_nX](TU)^n=\Big(\exp\sum_n p_{nn}(TU)^n\Big)^{\chi(X)}\,,$$
which expresses the vertical classes $[Z_nX]$ in terms of the primitive horizontal classes $p_{nn}$. 
\end{pf}

\begin{cor}\label{mst}
Let $e\in H^\sep$ be a multiplicative separable theory.  Assume that $e$ has total degree zero, by which we mean that $\langle e,\alpha\rangle=0$, unless $\alpha\in H_\ast^\sep$ is of total degree zero.  Then we have
$$\sum_n\langle e,[Z_nX]\rangle T^n
=\Big(\exp\sum_n \langle e,p_{nn}\rangle T^n\Big)^{\chi(X)}
=\Big(\sum_n \langle e,q_{nn}\rangle T^n\Big)^{\chi(X)}\,.$$
This reduces the calculation of the numerical invariants $\langle e,[Z_nX]\rangle$ for all curves $X$, to considering the families  $\O(1)\to\pp^n$ with the $n$-fold zero section.
\end{cor}

\subsection{Consequences II}

We will now deal with the case $d$ arbitrary.  Instead of using   $T_{\ol Z}$, which is not a vector bundle any longer, we will use inertial theories, see Example~\ref{inertia}.

\begin{lem}\label{minus} 
Let $P$ be a (formal)  multiplicative (or group-like) symmetric function, and $\ul P$ its associated inertial theory.  Let $X\to \ast$ be a proper $d$-fold.  Then 
$$\sum_{n=0}^\infty \langle \ul P,[Z_nX]\rangle T^n=(1-T)^{-\langle P(T_X),[X]\rangle}\,.$$
\end{lem}

\begin{thm}\label{iner}
Let $P$ be a multiplicative or group-like  symmetric function.  Let $\ul P$ be the inertial enumerative theory that $P$ defines. Then we have, for every $n$
$$\sum_{\vec m} \langle \ul P,p_{n\vec m}\rangle U_1^{m_1}\ldots U_d^{m_d}
=\frac{1}{n} (U_1\ldots U_d)^{n-1}P(U_1)\ldots P(U_d)\,.$$
Equivalently, 
$$\sum_{\vec x}\langle \ul P,p_{n,\vec x+n-1}\rangle U_1^{x_1}\ldots U_d^{x_d}
=\frac{1}{n}  P(U_1)\ldots P(U_d)\,.$$
\end{thm} 

\begin{thm}\label{sure}
We have
$$\sum_{n\geq0} [Z_nX]T^n=\exp\sum_{|\vec x|=d}\langle \gamma^{\vec x}, [X]\rangle\sum_{n>0} p_{n,\vec x+n-1} T^n\,.$$
This gives the following universal formula for $\langle e,[Z_nX]\rangle$, for any multiplicative separated theory $e$: 
$$\sum_{n\geq0} \langle e,[Z_nX]\rangle T^n=\exp\sum_{|\vec x|=d}\langle \gamma^{\vec x}
, [X]\rangle\sum_{n>0} \langle e,p_{n,\vec x+n-1}\rangle T^n\,.$$
It reduces the computation of all $\langle e,[Z_nX]\rangle$ to the computation of a few Chern numbers of $X$, and evaluating e on the equivariant multiple points $\langle e,q_{n\vec m}\rangle$. 
\end{thm}

\begin{pf}
As $\sum_n [Z_nX]T^n$ is group-like, its logarithm is a linear combination of primitives $p_{n\vec m}$
\begin{align*}
\log\sum_{n\geq0} [Z_nX]T^n=&
\sum_{n>0} T^n\sum_{|\vec m|=dn} \tilde a_{n\vec m} p_{n\vec m}\\
=&\sum_{n>0} T^n\sum_{|\vec x|=d} a_{n\vec x} p_{n,\vec x+n-1}
\end{align*}
Primitives $p_{n\vec m}$ with $|\vec m|\not=dn$ do not occur, because the left hand side is homogeneous of degree zero with respect to total degree. Also, we can restrict to $\vec m\geq n+1$, so we introduce $\vec x=\vec m-n+1$, so that $\vec x\geq0$. 

We will try to determine as many of the $a_{n\vec x}$ as possible, by pairing this equation with a multiplicative inertial theory 
$\ul P$. 
$$\log\sum_n \langle \ul P,[Z_nX]\rangle T^n=\sum_n T^n\sum_{|\vec x|=d} a_{n\vec x} \langle \ul P,p_{n,\vec x+n-1}\rangle\,.$$
By Lemma~\ref{minus} and Theorem~\ref{iner}, this gives, 
$$\langle P(T_X),[X]\rangle\sum_n\frac{1}{n}T^n=\sum_n T^n\sum_{|\vec x|=d}a_{n\vec x}\frac{1}{n}\Big(P(U_1)\ldots P(U_d)\Big)_{\text{degree $\vec x$-coefficient}}\,.$$
Or, for 
for every $n>0$ the equation
\begin{equation}\label{coeff}
  \sum_{|\vec x|=d}a_{n\vec x}\Big(P(U_1)\ldots P(U_d)\Big)_{\text{degree $\vec x$-coefficient}}=\langle P(T_X),[X]\rangle\,.
 \end{equation}
The generic multiplicative symmetric function is given by
$$P(U)=1+\sum_{j=1}^\infty t_j U^j\,,$$
for formal variables $t_1,t_2,\ldots$. Using this symmetric function we get, for the left hand side of (\ref{coeff}):
$$\sum_{|\vec x|=d}a_{n\vec x}\Big(\prod_{i=1}^d (1+t_1U_i+t_2U_i^2+\ldots)\Big)_{\text{degree $\vec x$-coefficient}}=\sum_{|\vec x|=d}a_{n\vec x} t_{x_1}\ldots t_{x_d}\,,$$
and for the right hand side
\begin{align*}
\langle \prod_{i=1}^d P(\gamma_i(T_X)),[X]\rangle=&
\langle \prod_{i=1}^d (1+t_1\gamma_i(T_X)+t_2\gamma_i^2(T_X)\ldots),[X]\rangle\\
=&\sum_{|\vec x|=d}t_{x_1}\ldots t_{x_d}\langle \prod_{i=1}^d \gamma_i(T_X)^{x_i}
,[X]\rangle\,.
\end{align*}
Equating the two sides, we see that
$$\sum_{|\vec x|=d}a_{n\vec x} t_{x_1}\ldots t_{x_d}=\sum_{|\vec x|=d}t_{x_1}\ldots t_{x_d}\langle \prod_{i=1}^d \gamma_i(T_X)^{x_i}
,[X]\rangle\,.$$
If we now set the $t_{x_i}$ equal to 1, we obtain
$$\sum_{|\vec x|=d} a_{n\vec x}=\sum_{|\vec x|=d}\langle \gamma^{\vec x},[X]\rangle\,.$$
This finishes the proof.
\end{pf}

\begin{thm} 
For any multiplicative separated theory $e$, the expression
$$\gamma\log \sum_{ n,\vec m}\langle e,q_{n\vec m}\rangle\gamma^{\vec m}\Big(\frac{T}{\gamma}\Big)^n$$
(where $\gamma_1,\ldots,\gamma_d$ are formal Chern roots, and $\gamma=\gamma_1\ldots \gamma_d$) has no poles, i.e., it is an element of $\qq[\gamma_1,\ldots,\gamma_d][[T]]$. Therefore, the following makes sense:\comment{I think we cannot restrict to $\vec x\geq0$ as the formula is written, because when rearranging, taking into account the relations, the formula will change  ALSO: it seems that when using the $p$s, we can restrict to $\vec x\geq0$, because when pairing  the primitives with inertial multiplicative theories we get 0.}
$$\log\sum_{n\geq0} \langle e,[Z_nX]\rangle T^n=\int_{[X]}\gamma\log \sum_{ n,\vec m}\langle e,q_{n\vec m}\rangle\gamma^{\vec m}\Big(\frac{T}{\gamma}\Big)^n\,.$$
It expresses the $\langle e,[Z_nX]\rangle$ directly  in terms of the Chern numbers of $X$ and the equivariant enumerative invariants $\langle e,q_{n\vec m}\rangle$. 

This formula shows how to convert the generating function for the equivariant invariants $\langle e,q_{n\vec m}\rangle$ into the generating function for the proper invariants $\langle e,[Z_nX]\rangle$.
\end{thm}
\begin{pf}
The expression under consideration is equal to 
$$\sum_{n,\vec x}\langle e,p_{n,\vec x+n-1} \rangle \gamma^{\vec x}T^n\,.$$
This follows directly from the definitional relationship between the $p_{n\vec m}$ and the $q_{n\vec m}$. This expression potentially extends over $\vec x$ with negative entries, but they can all be expressed in terms of those with $\vec x\geq0$. As $e$ is an enumerative theory, all relations among the $q_{n\vec m}$ also hold among the $\langle e,q_{n\vec m}\rangle$.  Therefore, any terms with negative entries in an $\vec x$ can be removed. 
\end{pf}
\begin{rmk}
In fact, we have the formula even before evaluating against a multiplicative theory:
$$\log\sum_{n\geq0} [Z_nX] T^n=\int_{[X]}\gamma\log \sum_{ n,\vec m}q_{n\vec m}\gamma^{\vec m}\Big(\frac{T}{\gamma}\Big)^n\,.$$
It is a fact that the expression under the integral sign is a power series in $H_\ast[\gamma_1,\ldots,\gamma_d][[T]]$, but we don't know how to write it explicitly as such. So this formula is a bit implicit.  Once we evaluate against a theory $e$, we can usually see the expansion in non-negative powers of $\gamma$, which allows us to get explicit formulas. 
\end{rmk}

\subsection{An example}

Let us continue Example~\ref{hilbert}. Consider the threefold $\aaa^3$, and its Hilbert scheme $\Hilb^n\aaa^3$. The torus $T=\Gm^3$ acts on $\aaa^3$ in the standard way.  We get an induced action on $\Hilb^n\aaa^3$. The virtual fundamental class of $\Hilb^n\aaa^3$ descends to an equivariant virtual class $[\Hilb^n\aaa^3/T\to BT]\in H^0(\Hilb^n\aaa^3/T\to BT)$.  
Recall that $H^\ast(BT)=R(T)=\qq[\gamma_1,\gamma_2,\gamma_3]=\qq[\gamma]$ is the symmetric algebra on the universal Chern roots, and that all bivariant groups $H^\ast[X/T\to BT]$ are $\qq[\gamma]$-modules.  Let us denote by $\qq(\gamma)$ the quotient field of the polynomial ring $\qq[\gamma]$. 

Denote by $F_n$ the fixed locus (the set of fixed points) of the action of $T$ on $\Hilb^n\aaa^3$. Let $i:F_n\to \Hilb^n\aaa^3$ be the inclusion morphism.  By the equivariant localization theorem,  we have 
$$[\Hilb^n\aaa^3/T\to BT]= i_\ast \frac{1}{e(N_{F/H})}\,,$$
once we tensor $H^\ast(\Hilb^n\aaa^3\to BT)$ (over $\qq[\gamma]$) with $\qq(\gamma)$. Here $e(N_{F/H})$ is the Euler class of the  virtual normal bundle to the inclusion $i:F_n\to \Hilb^n\aaa^3$. This Euler class is contained in $H^\ast (F_n\times BT)=\bigoplus_{x\in F_n} \qq[\gamma]$. Once tensored with $\qq(\gamma)$, it is an element of $\bigoplus_{x\in F_n}\qq(\gamma)$, and is an invertible element of degree $0$ there. 

Let $E=\O(1,0,0)+\O(0,1,0)+\O(0,0,1)$ be the tautological rank 3 bundle on $\pp^{\vec m}=\pp^{m_1}\times\pp^{m_2}\times\pp^{m_3}$. There is a canonical morphism  $\pp^{\vec m}\to BT$, such that $E$ is the pullback of the universal split rank 3 bundle. The relative Hilbert scheme of $E$ over $\pp^{\vec m}$ fits into a cartesian diagram
$$\begin{tikzcd}
\Hilb^n E\ar[r,"f"]\ar[d]& \ol Z_n E\ar[d]\ar[r] &\pp^{\vec m}\ar[d,"\pi"]\ar[l,bend right,"0"']\\
\Hilb^n\aaa^3/T\ar[r,"f"] & \ol Z_n \aaa^3/T \ar[r] &BT\ar[l,bend right,"0"']
\end{tikzcd}$$
Chasing through this diagram, we see that
\begin{align*}
\langle f_\ast e^\vir(\Hilb),q_{n,\vec m}\rangle=&
\langle f_\ast [\Hilb^n E\to\pp^{\vec m}],0_\ast[\pp^{\vec m}]\rangle\\
=&\langle f_\ast [\Hilb^n \aaa^3\to BT],0_\ast\pi_\ast[\pp^{\vec m}]\rangle\\
\end{align*}
Using notation as in the diagram
$$\begin{tikzcd}
F_n/T\ar[r,"j"]\ar[d,"i"] & Z_n\aaa^3/T\ar[d,"\psi"]\\
\Hilb^n\aaa^3/T\ar[r,"f"] & \ol Z_n\aaa^3/T
\end{tikzcd}$$
we continue the calculation:
\begin{align*}
=&\langle f_\ast i_\ast\frac{1}{e(N_{F/H})},0_\ast\pi_\ast[\pp^{\vec m}]\rangle\\
=&\langle \psi_\ast j_\ast\frac{1}{e(N_{F/H})},0_\ast\pi_\ast[\pp^{\vec m}]\rangle\\
=&\langle  j_\ast\frac{1}{e(N_{F/H})},0_\ast\pi_\ast[\pp^{\vec m}]\rangle\\
=&\langle \frac{e(T_{Z_n\aaa^3})}{e(N_{F/H})},\pi_\ast[\pp^{\vec m}]\rangle
\end{align*}
So evaluating the Hilbert enumerative theory on the tautological homology class $q_{n\vec m}$ amounts to evaluating the equivariant Donaldson-Thomas invariant $\frac{1}{e(N_{F/H})}\in \qq(\gamma)$, multiplied by $e(T_{Z_n\aaa^3})$ to raise it into degree $6n$, on the homology class $\pi_\ast[\pp^{\vec m}]\in H_{6n}(BT)$. 

In \cite{MNOPII} it is proved that
$$\sum_{n=0}^\infty \frac{T^n}{e(N_{F_n/H})}= {M(-T)} ^{-\frac{(\gamma_1+\gamma_2)(\gamma_2+\gamma_3)(\gamma_3+\gamma_1)}{\gamma_1\gamma_2\gamma_3}}$$
Substituting $T$ with $\gamma_1\gamma_2\gamma_3T$, we deduce
$$\sum_{n=0}^\infty \frac{(\gamma_1\gamma_2\gamma_3T)^n}{e(N_{F_n/H^n})}={M(-\gamma_1\gamma_2\gamma_3 T)}^{-\frac{(\gamma_1+\gamma_2)(\gamma_2+\gamma_3)(\gamma_3+\gamma_1)}{\gamma_1\gamma_2\gamma_3}}$$
We have $e(T_{Z_n\aaa^3})=(\gamma_1\gamma_2\gamma_3)^n$, and so 
\begin{align*}
\sum_{n,\vec m}\langle f_\ast e^\vir(\Hilb),q_{n,\vec m}\rangle T^n U^{\vec m}
=&\sum_{n,\vec m}\langle \frac{e(T_{Z_n\aaa^3})}{e(N_{F/H})},\pi_\ast[\pp^{\vec m}]\rangle  T^n U^{\vec m}\\
=&\sum_{\vec m} \langle 
 {M(-\gamma_1\gamma_2\gamma_3 T)}^{-\frac{(\gamma_1+\gamma_2)(\gamma_2+\gamma_3)(\gamma_3+\gamma_1)}{\gamma_1\gamma_2\gamma_3}}  ,\pi_\ast[\pp^{\vec m}]\rangle  U^{\vec m}\\
 =& 
 {M(-U_1U_2U_3 T)}^{-\frac{(U_1+U_2)(U_2+U_3)(U_3+U_1)}{U_1U_2U_3}} 
\end{align*}
This gives also
$$\sum_{n,\vec m}\langle f_\ast e^\vir(\Hilb),p_{n,\vec m}\rangle T^n U^{\vec m}=
-\frac{(U_1+U_2)(U_2+U_3)(U_3+U_1)}{U_1U_2U_3} \log M(-U_1U_2U_3T)$$
We rewrite this as:
$$\sum_{n,\vec x}\langle e,p_{n,\vec x+n-1}\rangle T^n U^{\vec x}=
-(U_1+U_2)(U_2+U_3)(U_3+U_1) \log M(-T)$$
Now we substitute $U_i$ with $\gamma_i$ and pair with $[X]$ 
$$\sum_{n,\vec x}\langle e,p_{n,\vec x+n-1}\rangle  \langle \gamma^{\vec x},[X]\rangle T^n=
-\langle (\gamma_1+\gamma_2)(\gamma_2+\gamma_3)(\gamma_3+\gamma_1),[X]\rangle \log M(-T)$$
Now applying Theorem~\ref{sure}, we get
$$\sum_{n\geq0}\langle e,[Z_nX]\rangle T^n=M(-T)^{\langle c_3- c_1c_2,[X]\rangle}$$
This agrees with the formula proved for example in \cite{cobordism}. In the Calabi-Yau case ($c_1=0$) it reduces to $M(-T)^{\chi (X)}$.



\end{document}